\documentclass{amsart}

\usepackage{amssymb, amsmath, amsthm, amsfonts}
\usepackage{hyperref}
\usepackage{setspace}

\usepackage[all]{xy}
\usepackage{graphicx}

\usepackage{ifthen}

\makeatletter
\newtheorem*{rep@theorem}{\rep@title}
\newcommand{\newreptheorem}[2]{%
\newenvironment{rep#1}[1]{%
 \def\rep@title{#2 \ref{##1}}%
 \begin{rep@theorem}}%
 {\end{rep@theorem}}}
\makeatother

\newtheorem{theorem}{Theorem}[section]
\newreptheorem{theorem}{Theorem}
\newtheorem{lemma}[theorem]{Lemma}
\newreptheorem{corollary}{Corollary}
\newtheorem{proposition}[theorem]{Proposition}
\newtheorem{corollary}[theorem]{Corollary}

\theoremstyle{definition}
\newtheorem{remark}[theorem]{Remark}
\newtheorem{remarks}[theorem]{Remarks}
\newtheorem{definition}[theorem]{Definition}

\newtheorem*{conjecture}{Conjecture}
\newtheorem{definitions}[theorem]{Definitions}

\def\beq{\begin{eqnarray*}}
\def\eeq{\end{eqnarray*}}

\def\R{\mathbb{R}}
\def\Z{\mathbb{Z}}

\def\incl{\hookrightarrow}
\def\to{\rightarrow}

\def\eps{\varepsilon}

\def\x{\times}

\def\phi{\varphi}

\def\Emb{\mathrm{Emb}}

\def\Aut{\mathrm{Aut}}

\def\HD{\mathcal{HD}}
\def\HW{\mathcal{HW}}
\def\HV{\mathcal{HV}}
\def\HC{\mathcal{HC}}
\def\HT{\mathcal{HT}}
\def\LT{\mathcal{LT}}
\def\LD{\mathcal{LD}}
\def\LW{\mathcal{LW}}
\def\LV{\mathcal{LV}}
\def\LC{\mathcal{LC}}

    \title[Integrals for Milnor invariants of string links]{Milnor invariants of string links, trivalent trees, and configuration space integrals}
    
    \author{Robin Koytcheff}
    \address{Department of Mathematics, University of Louisiana at Lafayette, Lafayette, LA}
    \email{koytcheff@louisiana.edu}
    \urladdr{people.math.umass.edu/~koytcheff}
    
    \author{Ismar Voli\'c}
    \address{Department of Mathematics, Wellesley College, Wellesley, MA}
    \email{ivolic@wellesley.edu}
    \urladdr{ivolic.wellesley.edu}

\thanks{The first author was supported in part by a PIMS Postdoctoral Fellowship.  The second author was supported in part by the Simons Foundation.}


\begin{document}
\maketitle

\begin{abstract}  
We study configuration space integral formulas for Milnor's homotopy link invariants, showing that they are in correspondence with certain linear combinations of trivalent trees.
Our proof is essentially a combinatorial analysis of a certain space of trivalent ``homotopy link diagrams'' which corresponds to all finite type homotopy link invariants via configuration space integrals.
An important ingredient is the fact that configuration space integrals take the shuffle product of diagrams to the product of invariants.  
We ultimately deduce a partial recipe for writing explicit integral formulas for products of Milnor invariants from trivalent forests.  
We also obtain cohomology classes in spaces of link maps from the same data.

\end{abstract}





\tableofcontents

\parskip=4pt
\parindent=0cm

\section{Introduction}

This article studies the relationship between two concepts.  The first is the link invariants defined by Milnor in his senior thesis in 1954 \cite{Milnor1954}.   Some of these are invariant under not just isotopy but also \emph{link homotopy},
which essentially means that they  detect linking of different strands but ignore knotting of individual strands.  Milnor considered closed links, for which these invariants are only defined modulo greatest common divisors of lower-order invariants.  In 1990, Habegger and Lin \cite{HabeggerLin} succeeded in classifying links up to link homotopy (i.e.~homotopy links) by using string links,\footnote{We consider a space of \emph{long links} which is homeomorphic to the space of string links.}
for which the Milnor invariants are well defined integer invariants.  
(See Section \ref{Background} for the precise definitions of string links and link homotopy.)
Furthermore, Milnor invariants are \emph{finite type} (or \emph{Vassiliev}) invariants (see Section \ref{FiniteType}), and these invariants distinguish string links up to link homotopy \cite{HabeggerLin, BarNatanJKTR}.  


The second side of the story is configuration space integrals (a brief overview of which is given in Section \ref{Integrals}).  They were first used by Gauss to define the linking number of two closed curves and were generalized by Bott and Taubes \cite{BottTaubes} in the 1990s.  
D.~Thurston \cite{DThurstonABThesis} used them to construct finite type knot and link invariants, and Cattaneo, Cotta-Ramusino, and Longoni extended them to knots in higher dimensions to obtain information about the cohomology of spaces of knots \cite{CCRL-AGT}. 
A further development was undertaken in \cite{KMV} where the authors extended configuration space integrals to long links and used them to construct all finite type invariants of this space. The motivation for the results in this paper in fact comes from the work in \cite{KMV}.  In short, as mentioned above, Milnor invariants are in particular of finite type, so some of the integrals 
constructed in \cite{KMV} are thus Milnor invariants.  It is our goal here to elucidate that relationship.

An important auxiliary ingredient for configuration space integrals are trivalent diagrams (or graphs), which are combinatorial objects and which form a part of a cochain complex (see Section \ref{Diagrams}).  
More specifically, each finite type invariant is given by a sum of integrals, and this linear combination of integrals corresponds to a linear combination of diagrams (see Section \ref{Integrals} for an overview).
The subspace of diagrams which correspond to \emph{link homotopy} invariants of string links was first found by Bar-Natan \cite{BarNatanJKTR} and further developed (in terms of the cochain complex and configuration space integrals) in \cite{KMV}.   It is these diagrams that will ultimately provide the connection between Milnor invariants of string links and configuration space integrals.

Combinatorial formulas for Milnor homotopy invariants were studied before by Mellor \cite{MellorMilnorWeight} who describes them recursively in terms of \emph{chord diagrams} (which are a special kind of trivalent diagrams).  This by itself does not give a clear conceptual picture, in part because Mellor's construction involves only the space of chord diagrams whereas configuration space integrals require enlarging this space to include all trivalent diagrams.    
Our results not only provide the diagram combinatorics for Milnor homotopy invariants, but in addition give an explicit way to construct them, up to lower order invariants, from certain trivalent trees using configuration space integrals.

Earlier related results are those of Bar-Natan \cite[Theorem 3]{BarNatanJKTR} and Habegger and Masbaum \cite[Theorem 6.1, Proposition 10.6]{HabeggerMasbaum}, who in fact established a more conceptual relationship between Milnor invariants and a space of trivalent trees modulo the IHX relation.  
(This relationship was perhaps expected, given the relationship between trivalent trees and Lie bracket expressions and the fact that Milnor invariants come from the lower central series of the link group.)
The work of Habegger and Masbaum gives a formula \cite[Figures 7, 8]{HabeggerMasbaum} in terms of the Kontsevich integral \cite{ChmutovDuzhin, BarNatanTopology, KontsevichVassiliev}, at least for the first non-vanishing Milnor invariant of a string link.  However, this formula appears to be at the level of the associated graded space and thus does not completely capture entire configuration space integral expressions.  
Moreover, Habegger and Masbaum, as well as Bar-Natan, consider a quotient of the space diagrams (essentially homology) whereas we view configuration space integrals as naturally indexed by the dual subspace (essentially cohomology). 
These dual spaces are not canonically isomorphic  (see Remark \ref{Dualizing}(a)).
Our purpose is to organize the conceptual relationship between Milnor invariants and trivalent trees with a view towards configuration space integral formulas for those invariants.  
A compelling motivation for 
providing explicit integral formulas is 
their potential application in geometric contexts, such as 
\cite{Komendarczyk-Michaelides-Ropelength, KomendarczykVolic}.  




Our approach is essentially a combinatorial analysis of the space of cocycles of trivalent diagrams using the shuffle product.  This product appears in the work of Cattaneo, Cotta-Ramusino, and Longoni \cite{CCRL-JKTR} and in the works of Turchin and Lambrechts \cite{Victor-Bialgebra, Victor-Pascal}.
Its main advantage is that it corresponds via configuration space integrals to the product of invariants.  
Thus it seems advantageous to redevelop the relationship between invariants and trees directly in this setting, rather than via ubiquitous dualizing of previous results of other authors.  
Using the shuffle product also allows us to relate Milnor homotopy invariants to trivalent trees without using any further structure on the space of diagrams.
For example, stacking long links defines a coproduct on diagrams compatible with the shuffle product, but
 we will not use such a coproduct in obtaining our results (see Remark \ref{Dualizing}(b)).
 
In this work, we consider only \emph{Milnor homotopy invariants}, and throughout this paper, we will often mean such invariants even when we omit the term ``homotopy.''  However, in our forthcoming work with Komendarczyk on integrals for spaces of braids \cite{KKV-Braids-Integrals}, we will be able to deduce analogous statements about integral formulae for arbitrary Milnor invariants (which are invariants of concordance but not necessarily homotopy).
 

\subsection{Statement of main results}
We now 
outline our main results; the reader may consult Definitions \ref{WeightSystemsDef}, \ref{FilterByComponentsDef}, and \ref{ConnectedIHXDef} for more details on the following notions.  
Let $\LW$ (resp.~$\HW$) be the space of weight systems for string links (resp.~homotopy string links).  By a pairing on graphs, we view weight systems as (defect zero) graph cocycles.  
By the \emph{leading terms} of a graph cocycle, we mean the terms with the fewest connected components.
We view $\LW$ and $\HW$ as algebras via the shuffle product.  Let $S(\mathcal{J})$ (resp.~$S(\mathcal{T})$) be the symmetric algebra on linear combinations of connected diagrams (resp.~trees) satisfying the IHX conditions.   
Let $\LW^k_n$ denote the space of graph cocycles with $2n$ vertices with at least $k$ components.  Let $S^k(\mathcal{J})_n$ 
(resp.~$S^k(\mathcal{T})_n$) be the subspace of elements of the $k$-th symmetric power of $\mathcal{J}$ (resp.~$\mathcal{T}$) with $2n$ vertices.

\begin{reptheorem}{LWSubalgebra}
Taking the leading terms of a cocycle defines injective maps  $$\LW^k_n / \LW^{k+1}_n \incl S^k(\mathcal{J})_n.$$  
A choice of splitting of $\LW$ induces an injective algebra map $\LW \incl S(\mathcal{J})$ which takes the shuffle product on $\LW$ to multiplication in $S(\mathcal{J})$.
\end{reptheorem}

Let $\HV$ be the space of finite type invariants of homotopy string links; the product of invariants makes this space into an algebra.  We also see that there is a filtration of $\HV$ by subspaces defined by  products of at least $k$ Milnor invariants, denoted by $\HV^k_n$.  The restrictions to $\HW$ of the maps in Theorem \ref{LWSubalgebra} land in $S(\mathcal{T})$.  We use the canonical map $\HV \to \HW$, the filtrations by $\geq k$ components, and the Milnor invariants to check that these maps are isomorphisms and obtain a canonical splitting of $\LW$.


\begin{reptheorem}{AlgIsos}
The Milnor homotopy invariants give canonical isomorphisms of algebras $$\HV \cong \HW \cong S(\mathcal{T})$$ using the products of  invariants on $\HV$, the shuffle product on $\HW$, and the usual product on $S(\mathcal{T})$.  
\end{reptheorem}

This result recovers the correspondence between Milnor invariants and trivalent trees (in the setting of homotopy invariants).  Unwinding the definitions of the filtration quotients and maps on them yields the following partial recipe for configuration space integral formulas.

\begin{repcorollary}{PartialRecipe}
Suppose $\alpha$ is a linear combination of (Lie-oriented) trivalent forests (with distinct labels on the leaves) satisfying the IHX conditions.
Then one can obtain a product of Milnor homotopy invariants
 by 
completing $\alpha$ to a cocycle $\alpha'$ 
by adding terms with more connected components 
and taking the configuration space integral associated to $\alpha'$.
All choices of the terms with more components yield the same invariant up to products of larger numbers of Milnor invariants.
\end{repcorollary}

In summary, our combinatorial analysis of the space of trivalent diagrams in terms of Milnor invariants has several main consequences.  
The first is to provide alternative proofs of certain facts about Milnor invariants, using mainly the shuffle product of diagrams.
The second is to give configuration space integral formulas  for Milnor invariants, up to products of lower-order invariants.  
Finally, since the same combinatorics of diagrams gives cohomology classes in spaces of long links in any dimension, we obtain the following result.

\begin{reptheorem}{HigherCohomology}
There is a subalgebra of cohomology classes in $\mathrm{Link}(\coprod_{i=1}^m \R, \R^d)$ isomorphic to the polynomial algebra on Milnor homotopy invariants.  
\end{reptheorem}


\subsection{Organization of the paper}
In Section \ref{Background}, we review some background material on finite type invariants,  Milnor invariants, trivalent diagrams, and configuration space integrals.  In Section \ref{Examples}, we express the simplest few Milnor invariants in terms of configuration space integrals.  This illustrates the main result of this paper.  It also shows that the formulas  become quite complicated even for type 3 Milnor invariants of 4-component links.  In Section \ref{Results}, we state and prove the main result of this paper, Theorem \ref{FiltrationIsomorphisms}.  This involves filtrations on spaces of diagrams and on spaces of finite type invariants of homotopy string links, as well as the shuffle product.  Finally, in Section \ref{Cohomology}, we consider spaces of link maps of 1-manifolds in $\R^d$, $d\geq 4$, and deduce that the graph cocycles which yield Milnor invariants in the case $d=3$ yield nontrivial cohomology classes in these spaces.  

\subsection{Acknowledgments}
We thank the referee for a careful reading and useful comments.  
The first author thanks Ben Williams for a discussion about linear algebra.  

\section{Background}
\label{Background}

\subsection{Homotopy long links}

A \emph{long link} $L$ on $m$ \emph{components} or \emph{strands} is a smooth embedding of $m$ disjoint copies of $\R$ into $\R^d$, $d\geq 3$, which outside of $\coprod_{i=1}^m [-1,1]$ is given by a fixed affine-linear embedding on each component.  For technical reasons, we require the directions of the $m$ affine-linear embeddings to be distinct (see \cite[Section 2.1]{KMV}).  A \emph{homotopy long link} is defined similarly, except that we no longer require $L$ to be an embedding on $\coprod_{i=1}^m [-1,1]$; instead we only require it to be a smooth map such that the images of any pair of components are disjoint.  (Such a map is often also called a \emph{link map} of $\coprod_{i=1}^m \R$ into $\R^d$.)  

We denote the set of all long links as $\mathrm{Emb}(\coprod_{i=1}^m \R, \R^d)$ and the set of homotopy long links as $\mathrm{Link}(\coprod_{i=1}^m \R, \R^d)$.  We topologize these spaces as in \cite[Section 2.2]{KMV}.  A path of long links is an \emph{isotopy}, and a path of homotopy long links is a \emph{link homotopy}.  By abuse of terminology, we also use ``homotopy long link'' to refer to a path component (i.e.~link homotopy class), rather than an element, of this space (just as a ``knot'' may refer to either an embedding or its isotopy class).  We abbreviate $\pi_0(\mathrm{Emb}(\coprod_{i=1}^m \R, \R^3))$ as $\mathcal{L}_m$ and $\pi_0(\mathrm{Link}(\coprod_{i=1}^m \R, \R^3))$ as $\mathcal{H}_m$.  Since every homotopy long link is isotopic to an embedded long link, there is a canonical quotient map $\mathcal{L}_m \to \mathcal{H}_m$.

The stacking of long links makes $\mathcal{L}_m$ and $\mathcal{H}_m$ into monoids.  Moreover, $\mathcal{H}_m$ is a group and is nilpotent.  It is a quotient of the pure braid group and more specifically an iterated semidirect product of reduced free groups (see \cite[Section 1]{HabeggerLin} or \cite{LinPowerSeries}).  Although we will not use this structure, it provides one possible approach to studying finite type invariants of homotopy long links.

\subsection{Finite type invariants of homotopy long links}\label{FiniteType}
An \emph{isotopy invariant} (resp.~\emph{homotopy invariant}) of $m$-component long links is a map $\mathcal{L}_m \to \R$ (resp.~$\mathcal{H}_m \to \R$).  Thus a homotopy invariant is also an isotopy invariant.
Any invariant extends to \emph{singular long links}, i.e.~long links with finitely many intersection points,  
by inductively using the Vassiliev skein relation, where the arrows denote the orientation(s) of the link component(s):

\[
v\left(\raisebox{-0.8pc}{\includegraphics[height=2pc]{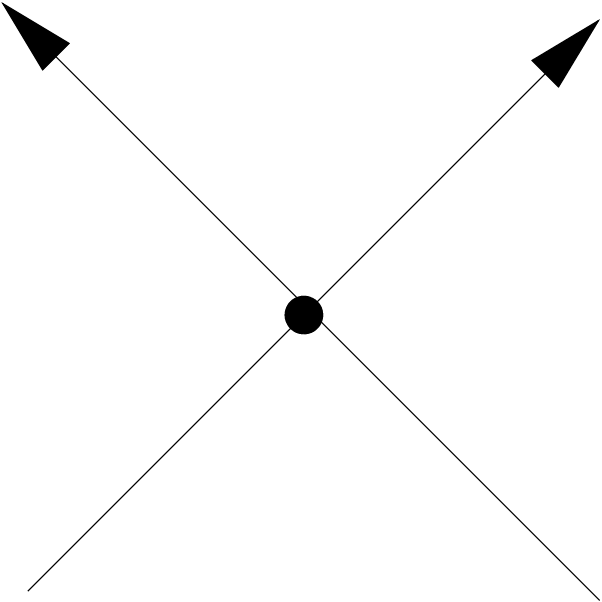}}\right) = v\left(\raisebox{-0.9pc}{\includegraphics[height=2pc]{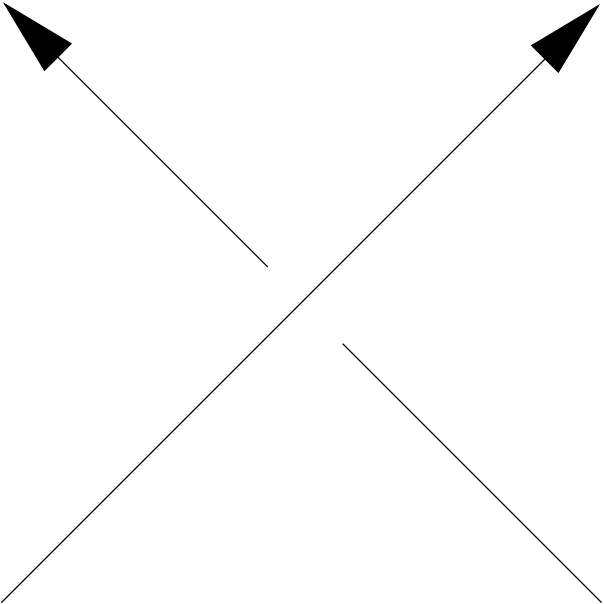}}\right) - v\left(\raisebox{-0.9pc}{\includegraphics[height=2pc]{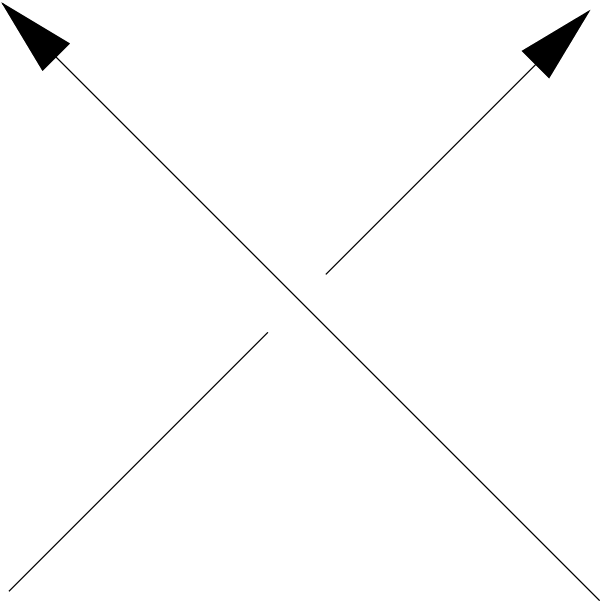}}\right)
\]

An invariant is called \emph{finite type of \emph{type} $n$} (or of \emph{order} $n$) if it vanishes on all links with more than $n$ such intersection points.  
Let $\LV_n$ (resp.~$\HV_n$) denote the $\R$-vector space of type $n$ finite type invariants of long links in $\mathcal{L}_m$ (resp.~homotopy long links in $\mathcal{H}_m$).\footnote{One can also define $\HV_n$ as functions on the quotient of the group ring of $\mathcal{H}_m$ by the $n$-th power of the augmentation ideal, as noted for example in \cite[Theorem 30]{WillertonThesis}.}
Milnor's homotopy invariants of links (which we will review in Section \ref{MilnorInvs}) are finite type invariants of $\mathcal{H}_m$.  In particular, a Milnor invariant involving $n+1$ strands of the link is of type $n$.
Using work of Habegger and Lin \cite{HabeggerLin}, Bar-Natan \cite{BarNatanJKTR} showed that finite type invariants classify homotopy long links.

The product of two finite type invariants $v_1$ and $v_2$ is given simply by $$v_1v_2(L):=v_1(L)v_2(L).$$
The proof of the following can be found, for example, in \cite[p.~75]{Chmutov-Duzhin-Mostovoy}.

\begin{lemma}
\label{ProductOfFiniteTypes} 
The product of two finite type link invariants of types $n$ and $p$ is a finite type link invariant of type $n+p$.  
\qed
\end{lemma}


\subsection{Trivalent diagrams}
\label{Diagrams}
Finite type invariants can be described in completely combinatorial terms.  Trivalent diagrams help establish this correspondence, as well as more general results.

\begin{definition}
\label{TrivalentDiagrams}
A \emph{(long) link trivalent diagram} of order $n$ consists of $m$ segments, $2n$ vertices, some of which lie on the segments, and some \emph{edges} between vertices.  The vertices which lie on segments are called \emph{segment vertices}, while the others are called \emph{free vertices}.  For a diagram $\Gamma$, we denote the set of these by $V_{seg}(\Gamma)$ and $V_{free}(\Gamma)$ respectively, and let $$V(\Gamma) = V_{seg}(\Gamma) \sqcup V_{free}(\Gamma).$$
We require 
\begin{itemize}
\item
each vertex to be trivalent, where both the edges and segments contribute to valence; and 
\item
each vertex to be connected by a path of edges to some segment.
\end{itemize}
A \emph{homotopy (string) link trivalent diagram} is a trivalent diagram in which
no two vertices on the same segment are joined by a path of edges.  
A \emph{chord diagram} is a trivalent diagram with no free vertices; hence a chord diagram of order $n$ has exactly $n$ chords.  
\end{definition}



\begin{figure}[h!]
\begin{center}
\includegraphics[height=6.5pc]{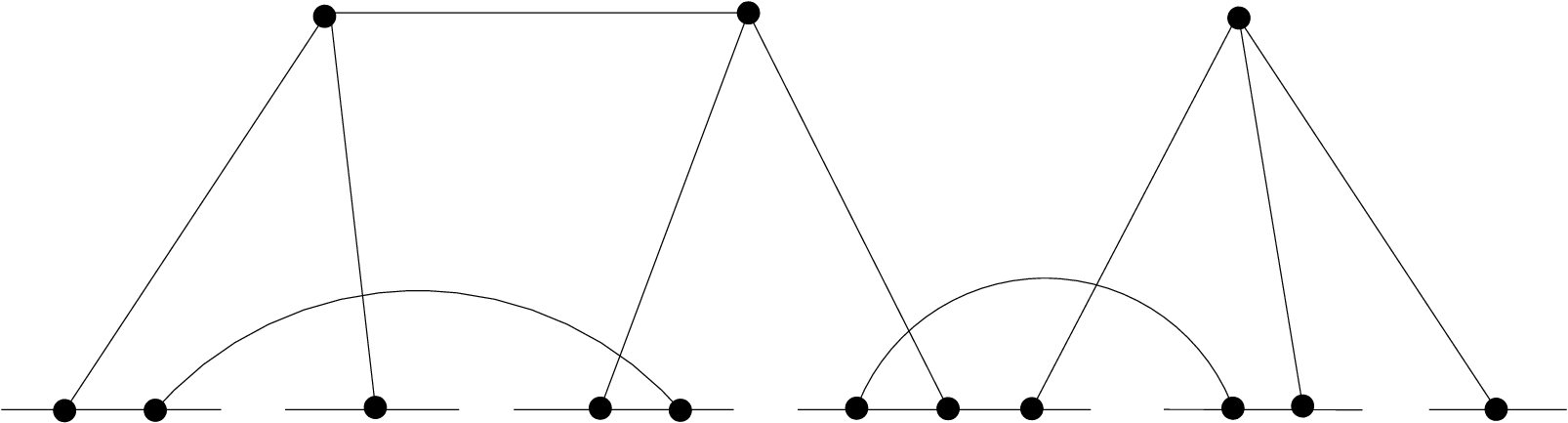}
\end{center}
\caption{A trivalent diagram of order $n=7$, where the number of segments is $m=6$.}
\end{figure}
We consider trivalent diagrams up to homeomorphisms which preserve the orientations and labels of the segments.  
Denote the $\R$-vector space of chord diagrams of order $n$ by $\LC_n$ and the subspace of homotopy chord diagrams as $\HC_n$.

A trivalent diagram $\Gamma$ is enough to determine a configuration space integral and hence a class in $H^*\Emb(\R, \R^d)$, but only up to sign.  To consistently determine signs, we need orientations. 
This means that we add certain decorations to diagrams, but diagrams differing only in these decorations are set equal, up to a sign.  See for example \cite[Definition 3.11]{KMV}.
So an orientation is an equivalence class of decorations, and thus the vector space of oriented diagrams is isomorphic to the vector space of the unoriented diagrams defined above (though the isomorphism is canonical only up to signs).  
The required type of orientation differs depending on the parity of the dimension $d$ of the ambient Euclidean space.  We first address the case of odd $d$, which is our primary focus.  We alternately use two equivalent types of orientations:
\begin{itemize}
\item
One way to orient diagrams is to order the vertices and orient each edge.
If $\Gamma$ and $\Gamma'$ differ by a permutation $\sigma$ of the vertex labels, then $\Gamma$ is declared to be equal to $(-1)^{\mathrm{sign}(\sigma)}\Gamma'$.  If $\Gamma$ and $\Gamma'$ differ by the orientation of one edge, then $\Gamma=-\Gamma'$.  The advantage of this orientation is that it is most directly related to configuration space integrals.  We thus call this an \emph{integration orientation}.
\item
A second way, often called a \emph{Lie orientation} \cite{KuperbergThurston},
is to give a cyclic ordering of the half-edges emanating from each vertex. 
Typically, one specifies such an orientation using a planar embedding of $\Gamma$, ordering the half-edges at each vertex counter-clockwise.  If $\Gamma$ and $\Gamma'$ differ by a permutation $\sigma$ of half-edges at one vertex, then $\Gamma$ is declared to be equal to $(-1)^{\mathrm{sign}(\sigma)}\Gamma'$.  
Note that if $\Gamma$ is a chord diagram, such an orientation consists of no data at all.
The Lie orientation has the advantage that it requires only a planar embedding of a diagram and no labels or arrows.  In this context, the orientation relations are often called \emph{antisymmetry (AS) relations}.
\end{itemize}
Either of the two types of orientation canonically determines the other, as explained in \cite[p.~4]{KuperbergThurston} and \cite[Proposition B.1]{DThurstonABThesis}.  
We will consider arbitrary $d$ in Section \ref{Cohomology}.
\begin{itemize}
\item
For even $d$, an orientation can be given by an ordering of the edges and an ordering of the segment vertices.  If $\Gamma$ and $\Gamma'$ differ by a permutation $\sigma$ of the edge-labels, then then $\Gamma$ is set equal to $(-1)^{\mathrm{sign}(\sigma)}\Gamma'$.  A similar relation is imposed if $\Gamma$ and $\Gamma'$ differ by a permutation of the segment vertex labels.
In this parity, this is also precisely the data needed to define a configuration space integral on the nose, not just up to sign.
\end{itemize}
Every diagram appearing in the expression for a Milnor invariant will be a forest, meaning that it has no closed loops of edges.  
Certain salient terms will be trees, meaning that any two vertices can be joined by a path of edges.  In the case of trees, as shown in the above two references, an orientation of odd type canonically determines an orientation of even type (and vice-versa).

\begin{definition}
Let $\LT_n$ denote the $\R$-vector space generated by link trivalent diagrams on $2n$ vertices, oriented as above, modulo the relations that diagrams with multiple edges between a pair of vertices or self-loops on free vertices are set to zero.  
Set $$\LT=\bigoplus_{n=0}^\infty \LT_n.$$  Similarly define $\HT_n$ and $\HT$ for the corresponding subspaces of homotopy link trivalent diagrams.
\end{definition}

\subsection{The cochain complex of diagrams}

The vector spaces $\LT_n$ (resp.~$\HT_n$) are part of a certain cochain complex $\LD$ (resp.~$\HD$) of link diagrams (resp.~homotopy link diagrams), which are not necessarily trivalent.
This complex will be used in Section \ref{Cohomology}.
\begin{definition}
Elements of $\LD$ (and $\HD$) are defined exactly as the elements of $\LT$ (and $\HT$), except that instead of requiring every vertex to have valence three, one requires that every vertex has valence \emph{at least} three.  Rather than just counting the number of vertices in a diagram, we define
\begin{itemize}
\item
 the \emph{defect} of a diagram $\Gamma$ to be $2|E(\Gamma)| - |V_{seg}(\Gamma)| - 3|V_{free}(\Gamma)|$; and
\item
the \emph{order} of a diagram $\Gamma$ to be $|E(\Gamma)| - |V_{free}(\Gamma)|$.
\end{itemize}
Diagrams in $\LD$ are oriented via the integration orientation as for $\LT$.  (The equivalent Lie orientation in the general case would require a modification to include an ordering of the even-valence vertices.)
Then let $\LD_{k,n}$ be the vector space of diagrams of defect $k$ and order $n$, and let 
$$\LD=\bigoplus_{k,n} \LD_{k,n}.$$  Define $\HD_{k,n}$ and $\HD$ similarly, but with the corresponding subspaces of homotopy link diagrams.
\end{definition}

Note that $\LT_n = \LD_{0,n}$ since trivalence is equivalent to having defect zero.  We equip 
$\LD$
with a differential $\delta$ which takes $\Gamma$ to a signed sum of graphs $\sum_i \Gamma_i$ where each $\Gamma_i$ is obtained by contracting either an edge or an \emph{arc} of $\Gamma$.  An arc is defined as part of a segment between two vertices.  The signs of the $\Gamma_i$ are determined using integration orientations on diagrams.  The differential is a map 
$$\delta \colon \LD_{k,n} \longrightarrow \LD_{k+1,n}.$$  The subspace $\HD$ is also a subcomplex of $\LD$, so we have maps $\HD_{k,n} \to \HD_{k+1,n}$ (see \cite[Definition 3.19]{KMV} for details).

The importance of the diagram complexes $\LD$ and $\HD$ comes from (co)chain maps \cite[Theorem 4.33]{KMV}
\begin{align}
\label{CochainMap}
I_{\mathcal{L}}: \LD_{k,n} \longrightarrow \Omega_{dR}^{n(d-3)+k} \mathrm{Emb}(\coprod_{i=1}^m \R, \R^d), & \ \ \ \ \mbox{ for } d>3; \\
\label{CochainMapH}
I_{\mathcal{H}}: \HD_{k,n} \longrightarrow \Omega_{dR}^{n(d-3)+k} \mathrm{Link}(\coprod_{i=1}^m \R, \R^d), & \ \ \ \ \mbox{ for } d\geq 3. 
\end{align}
Here $\Omega_{dR}^{n(d-3)+k}$ stands for de Rham cochains (i.e.~differential forms) of degree $n(d-3)+k$.  
These maps induce maps in cohomology.  Either map is given by associating to each diagram a configuration space integral.  
For $d=3$ one must add a certain ``anomaly correction term'' to $I_{\mathcal{L}}$ to make it a chain map.
These integrals will be reviewed in Section \ref{Integrals}.



\subsubsection{Cocycles in defect zero}
\label{DescribingCocycles}
Consider first the space $\LT_n^*$ spanned by the same diagrams as $\LT_n$, but viewed as dual to $\LT_n$ via the pairing
\begin{equation}
\label{PairingWithAutFactors}
\langle \Gamma, \Gamma' \rangle = \left\{
\begin{array}{ll} |\mathrm{Aut}(\Gamma)|, & \text{ if $\Gamma \cong \Gamma'$};  \\
0, & \text{ otherwise.}
\end{array}
\right.
\end{equation}
Let $STU$ denote all the relations on diagrams $S$, $T$, and $U$ of the form given in Figure \ref{STUFigure}, where the horizontal line is a link strand, the diagrams are identical outside of the pictured portions, and none of them has a multiple edge.  
\begin{figure}[h]
\[
\raisebox{-1.2pc}{\includegraphics[scale=0.15]{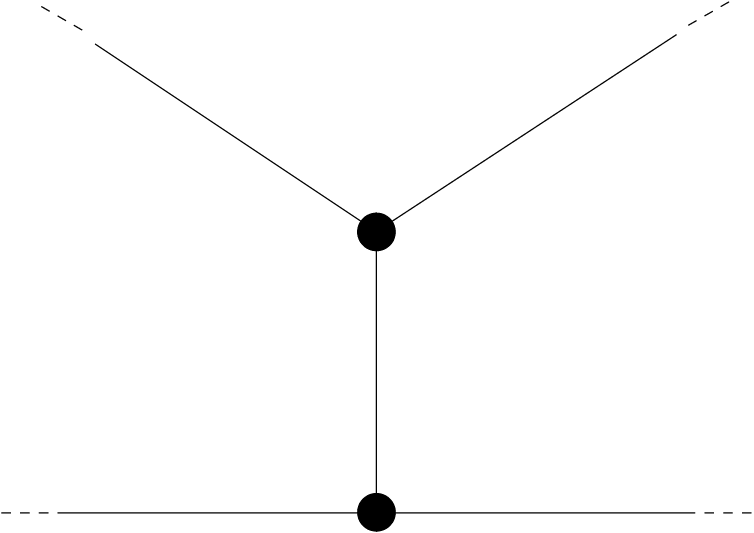}} \ -\ 
\raisebox{-1.2pc}{\includegraphics[scale=0.15]{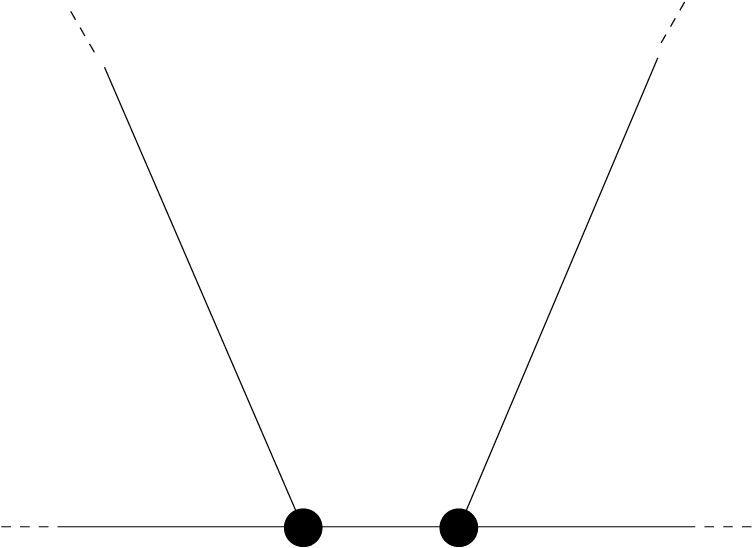}} \ +\  
\raisebox{-1.2pc}{\includegraphics[scale=0.15]{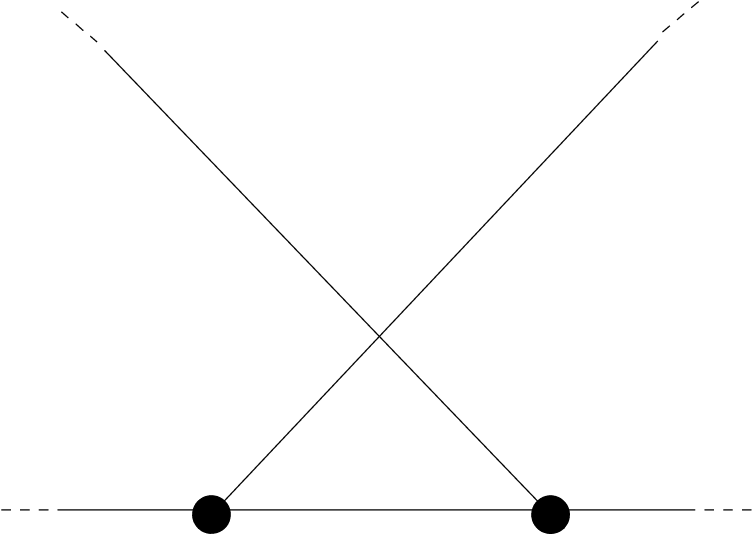}} =0
\]
\caption{The STU relation.}
\label{STUFigure}
\end{figure}

Similarly let $IHX$ denote all the relations on diagrams $I$, $H$, and $X$ of the form given in Figure \ref{IHXFigure}, where all the vertices are free vertices, the diagrams are identical outside of the pictured portions, and none of them has a multiple edge or a self-loop.  
\begin{figure}[h]
\[
\raisebox{-1.9pc}{\includegraphics[scale=0.15]{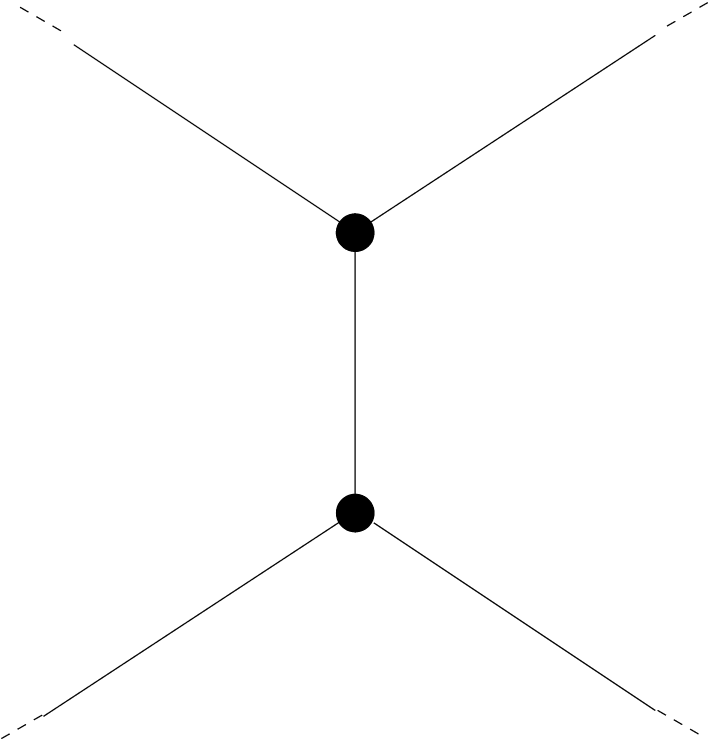}}\   -\ 
\raisebox{-1.6pc}{\includegraphics[scale=0.15]{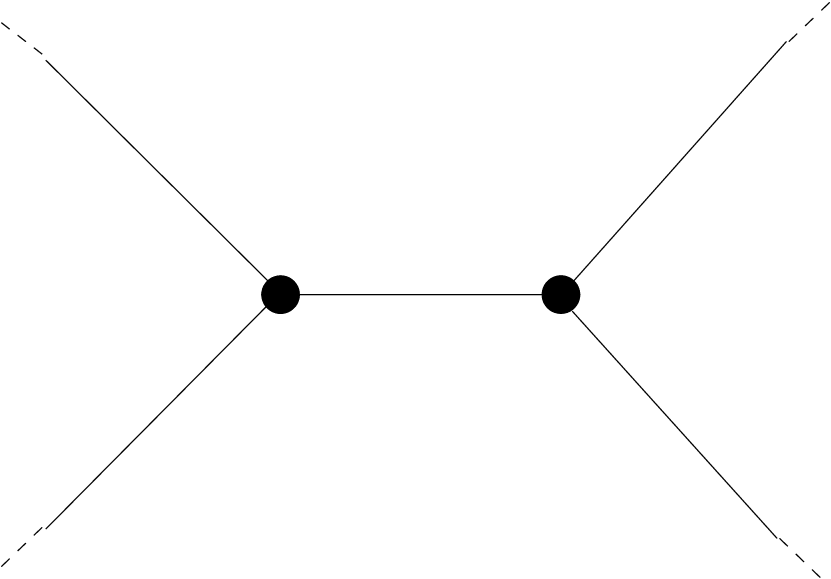}}\   +\  
\raisebox{-1.6pc}{\includegraphics[scale=0.15]{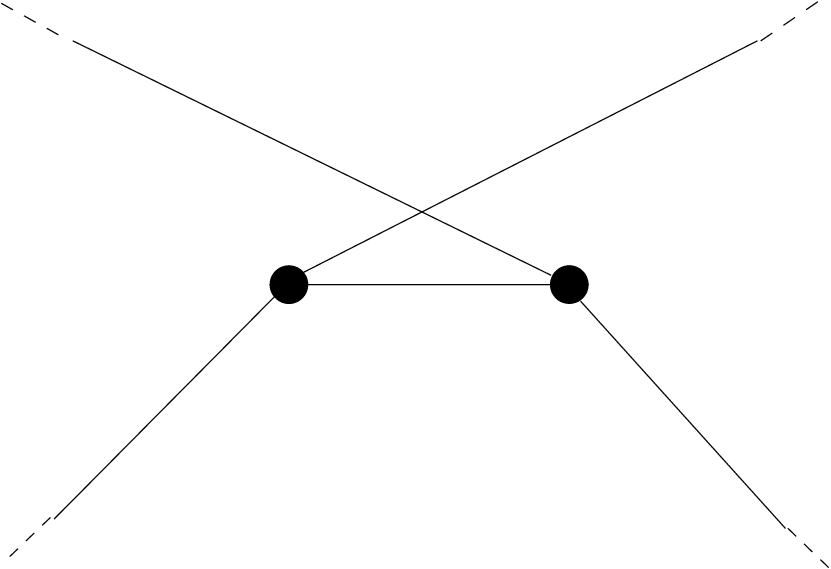}} = 0
\]
\caption{The IHX relation.}
\label{IHXFigure}
\end{figure}

In either case, the orientations on the diagrams are the Lie orientations given by the planar embeddings.

\begin{definition}
\label{WeightSystemsDef}
Define the space of \emph{weight systems} as $$\LW_n = (\LT_n^*/(STU, IHX))^* \subset \LT_n.$$  
\end{definition}

This space can be shown to coincide with the space of cocycles $Z(\LD_{0,n}) \subset \LT_n$ by considering the adjoint of the differential, i.e.~``blowing up'' a 4-valent vertex \cite[Proposition 3.29]{KMV}, \cite[Proposition 1]{ConantVogtmann:MoritaVanish}.


That is, the cocycles in $\LD_{0,n}$ are precisely the elements $\alpha$ which give 0 when paired with a diagram with a closed loop of edges, and which satisfy $\langle S-T+U, \alpha\rangle=0$ for all $S,T,U$ as in Figure \ref{STUFigure}.


\begin{proposition}\label{Aut=1}
For any homotopy link trivalent diagram $\Gamma$, $\mathrm{Aut}(\Gamma)=1$.
\end{proposition}
\begin{proof}
An automorphism of diagrams must fix all segment vertices.  
Furthermore, a homotopy link trivalent diagram has no closed loops of edges (due to the STU relation and the fact that homotopy link diagrams have no paths joining vertices on the same segment).  Thus every remaining vertex appears on some path (possibly many) that is the unique path between two segment vertices.
\end{proof}

\begin{remarks}
\label{RelationVsCondition}\ 
\begin{enumerate}  
\item[(a)] As a consequence of Proposition \ref{Aut=1}, the pairing in \eqref{PairingWithAutFactors} above reduces to the Kronecker pairing on homotopy link diagrams $\Gamma \in \HD$.  
Then 
we can view a cocycle (or weight system) as a linear combination of diagrams such that, for every triple of diagrams $S,T,U$ as in Figure \ref{STUFigure} (and every triple $I,H,X$ as in Figure \ref{IHXFigure}), the coefficients of these three diagrams sum to zero.  
While the relations in the \emph{quotient} $\LT _n^*/(STU, IHX)$ are often referred to as the STU and IHX \emph{relations}, we will use the term STU and IHX \emph{conditions} to describe the restriction that determines the \emph{subspace} of cocycles in $\LT_n$.  

\item[(b)] The above proposition fails for diagrams in $\mathcal{LD}$ that are not homotopy link diagrams.  For example, the diagram below has a nontrivial automorphism.
\bigskip

\begin{center}
\includegraphics[scale=0.18]{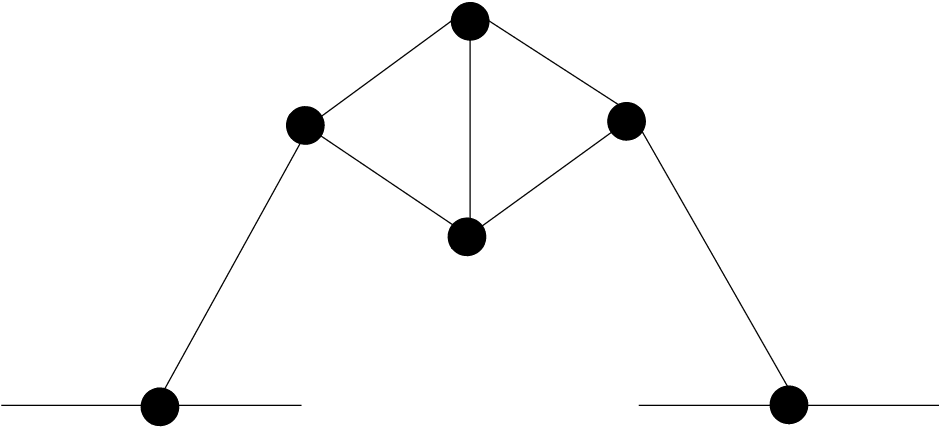}
\end{center}
\end{enumerate}
\end{remarks}

Considering the cochain map \eqref{CochainMapH} in the case of defect $k=0$ and ambient dimension $d=3$, we get a map 
\[
I_{\mathcal{H}}: \HW_n = Z(\HD_{0,n}) \longrightarrow \mathrm{H}^0(\mathrm{Link}(\coprod_{i=1}^m \R, \R^3))
\]
whose image lies in $\HV_n$.
Conversely, there is a canonical map\footnote{This canonical map is defined by viewing $\LW_n$ as $(\LC_n^*/4T)^*$ where $4T$ stands for the well known \emph{4-term relation}, and where $(\LC_n^*/4T)^* \cong (\LT_n^*/STU)^*$.  See \cite[Theorem 6]{BarNatanTopology} for a proof of the analogous isomorphism for knots.}
 $\HV_n \to \HW_n$ which factors as
\begin{equation}
\label{CanonicalMap}
\HV_n / \HV_{n-1} \longrightarrow \HW_n.
\end{equation}
It is not too difficult to show that \eqref{CanonicalMap} is well defined and injective.
\begin{theorem}[\cite{KMV}, Theorem 5.8]\label{FTFTI}
The composite of integration followed by the quotient 
\begin{equation}
\label{InverseInFTFTI}
\HW_n \overset{I_\mathcal{H}}{\longrightarrow} \HV_n \longrightarrow \HV_n / \HV_{n-1}
\end{equation}
is inverse to the canonical map \eqref{CanonicalMap}.
\end{theorem}
An analogous statement for finite type \emph{isotopy} invariants of links is true, provided one adds an anomaly correction term to $I_{\mathcal{L}}$.
Such an inverse map can also be obtained via the Kontsevich integral (see \cite{ChmutovDuzhin, BarNatanTopology} for the case of isotopy invariants and \cite[Proposition 10.4]{HabeggerMasbaum} for the special case of homotopy string link invariants).  


\subsubsection{The shuffle product}
There is a product on 
$\LD$
called the \emph{shuffle product}, defined as follows.  Given two diagrams $\Gamma_1$ and $\Gamma_2$ with $q_i$ and $r_i$ vertices on the $i$-th segment for $i=1,\ldots,m$, let $\sigma$ be an $m$-tuple of shuffles, that is, orderings of the $q_i+r_i$ vertices which respect the orderings in $\Gamma_1$ and $\Gamma_2$.  Such a $\sigma$ specifies a way of attaching the ``legs'' of $\Gamma_1$ and $\Gamma_2$ to the $m$ segments to form a new diagram $\Gamma_1 \bullet_\sigma \Gamma_2$.  
The integration orientation of $\Gamma_1 \bullet_\sigma \Gamma_2$ is given by leaving the labels of all the vertices from $\Gamma_1$ unchanged and raising by $|V(\Gamma_1)|$ the label of every vertex from $\Gamma_2$.  
One can check that this is equivalent to endowing $\Gamma_1 \bullet_\sigma \Gamma_2$ with the Lie orientation obtained by superimposing the planar diagrams for $\Gamma_1$ and $\Gamma_2$.  
For $d$ odd, the shuffle product is then defined as 
\[
\Gamma_1\bullet\Gamma_2 := \sum_\sigma \Gamma_1 \bullet_\sigma \Gamma_2.
\]
For $d$ even, an extra sign $\eps(\Gamma_1,\Gamma_2) \in \{\pm 1\}$ is needed.
%

\begin{proposition}
The shuffle product restricts to a product 
$$\HW_n \x \HW_p \to \HW_{n+p}.$$
\end{proposition}
\begin{proof}
This follows from the fact that the differential $\delta$ in the graph complex satisfies a Leibniz rule with respect to the shuffle product $\bullet$ (see Proposition 3.23 of \cite{KMV} or Proposition 3.1 of \cite{CCRL-JKTR} for the case of long knots).  Alternatively, one can use the description of defect zero cocycles (weight systems) in terms of the STU condition, as follows.  If (the coefficients of) $\alpha$ and $\beta$ satisfy every STU condition, then so does $\alpha \bullet \beta$: the shuffle product introduces no new edges of the sort in the $S$ diagram, while every new arc of the sort in the $T$ diagram is matched by a new arc of the sort in the $U$ diagram, and the two matching diagrams appear with the same sign.  
\end{proof}

The following lemma is an important ingredient in the proof of our main result, Theorem \ref{FiltrationIsomorphisms}.  For a proof, see for example, the proof of Proposition 4.29 in \cite{KMV}.
Recall the discussion of the product of finite type homotopy long link invariants from the end of Section \ref{FiniteType}.

\begin{lemma}
\label{ShuffleProductLemma}
Via the isomorphisms $\HW_n \cong \HV_n / \HV_{n-1}$ from Theorem \ref{FTFTI} -- see the maps \eqref{InverseInFTFTI} and \eqref{CanonicalMap} -- the shuffle product of diagrams corresponds to the product of finite type link invariants.
\qed
\end{lemma}

\begin{remark}
The shuffle product of link diagrams appears in earlier literature, such as \cite{CCRL-JKTR,Victor-Bialgebra,Victor-Pascal}.  
This statement seems to appear in a different guise in \cite[Exercise 3.10]{BarNatanTopology}.  There, invariants correspond to functionals on diagrams (weight systems), rather than the diagrams themselves, so the product of invariants corresponds to a product on such functionals, which is a \emph{coproduct} on diagrams.  For configuration space integral formulas, we find it more convenient to think in terms of products of diagrams. 
\end{remark}

\subsection{Configuration space integrals}
\label{Integrals}

We now review the generalization of the configuration space integrals of Bott and Taubes \cite{BottTaubes} to long links; this generalization appears in \cite{KMV}.  We will consider only \emph{embedded} long links, that is the space $\Emb(\coprod_1^m \R, \R^d)$, rather than $\mathrm{Link}(\coprod_1^m \R, \R^d)$.
This is because of the fact that, to study invariants of homotopy long links in $\R^3$, it suffices to pull back cohomology classes via the inclusion $\Emb \incl \mathrm{Link}$.  In fact,  $\mathcal{L}_m$ surjects onto $\mathcal{H}_m$, and the integration maps $I_{\mathcal{L}}$ and $I_{\mathcal{H}}$ agree when restricting the latter to homotopy diagrams $\HD$ and the former to embedded long links. 


Fix a trivalent diagram $\Gamma \in \LT$ with labeled vertices and oriented edges.  Recall from the discussion following Definition \ref{TrivalentDiagrams} that a Lie orientation on $\Gamma$ can be used to produce vertex labels and edge orientations (up to even permutations of the vertex labels and an even number of edge reversals).  Suppose $\Gamma$ has $q_i$ vertices on the $i$-th segment and $t$ free vertices.
Consider the pullback square
\begin{equation}\label{BTsquare}
 \xymatrix{
E[q_1,\ldots, q_m;t] \ar[r] \ar[d] & C_{q_1+\dotsb +q_m+t}[\R^d]\ar[d] \\
\Emb(\coprod_1^m \R, \R^d) \x  C_{q_1,\ldots,q_m} \left[\coprod_1^m \R\right]  \ar[r] & C_{q_1+\dotsb +q_m}[\R^d]}
\end{equation}
Here
\begin{itemize}
\item Each space $C_k[\R^d]$ is the Axelrod--Singer compactification of the configuration space of $k$ points in $\R^d$  (see e.g.~\cite[Section 4]{KMV}).
\item The right vertical map is given by forgetting the last $t$ points.  
\item The space $C_{q_1,\ldots,q_m} \left[\coprod_1^m \R\right]$ is a compactification of the configuration space of $q_1+\cdots+q_m$ points, $q_i$ of which are on the $i$-th copy of $\R$, which records all the relative rates of approach to infinity (see \cite[Section 4]{KMV}).  
\item The bottom horizontal map comes from the fact that an embedding of $X$ into $Y$ induces a map from configurations in $X$ to configurations in $Y$.  
\end{itemize}
The bundle
\[
E[q_1,\ldots,q_m;t] \longrightarrow \Emb(\coprod_1^m \R, \R^d)
\]
is given by the left vertical map in (\ref{BTsquare}) followed by projection onto the first factor.  Write $F=F[q_1,\ldots ,q_m;t]$ for its fiber, which is a compactified configuration space of $q_1+\cdots+q_m+t$ points in $\R^d$, $q_i$ of which lie on the $i$-th strand of the given link and $t$ of which can be anywhere in $\R^d$, including on the link.  This fiber is a (finite-dimensional) manifold with corners. 

There are maps
\[
\phi_{ij}:C_{q_1+\dotsb +q_m+t}[\R^d] \longrightarrow S^{d-1}
\]
given by the direction between the $i$-th and $j$-th configuration point.  
Via $\phi_{ij}$, one can pull back the volume form $\mathrm{sym}_{S^{d-1}}$ on $S^{d-1}$ to $C_{q_1+\dotsb +q_m+t}[\R^d]$ and then further to $E[q_1,\ldots, q_m;t]$.  Let $\theta_{ij}$ denote the resulting spherical form on $E[q_1,\ldots, q_m;t]$.  
For our fixed diagram $\Gamma$, take the product $\beta$ of all $\theta_{ij}$ such that vertices $i$ and $j$ are endpoints of an edge in $\Gamma$.  Then let $I_{\mathcal{L}}(\Gamma)$ be the integral of $\beta$ over  the fiber $F=F[q_1,\ldots,q_m;t]$ of the bundle above:
\begin{equation}\label{GeneralIntegral}
I_{\mathcal{L}}(\Gamma) = \int_{F}\beta = \int_{F[q_1,\ldots ,q_m;t]}\prod_{\text{edges $ij$ of $\Gamma$}}\theta_{ij}
\end{equation}

Thus $I_{\mathcal{L}}(\Gamma)$ is a differential form on $\Emb(\coprod_1^m \R, \R^d)$, of degree 0 in the case $d=3$.  
A linear combination of such forms is closed if and only if certain boundary contributions cancel (by Stokes' Theorem), and the such linear combinations are indexed precisely by cocycles in the graph complex. 
For further details, the reader may consult \cite[Section 4.4]{KMV}.

In \cite{KMV}, the authors showed how to extend configuration space integrals to homotopy long links, that is, link maps which are not necessarily embeddings.  
For details, see \cite[Section 4.2.3]{KMV}.  This setup for link maps reduces to the setup above in the case when the link map is an embedding.  
The modification of these integrals for homotopy long links provides the inverse in the link homotopy version of the Fundamental Theorem of Finite Type Invariants, given in Theorem  \ref{FTFTI}.


\subsection{Milnor's link homotopy invariants}\label{MilnorInvs}
For each tuple $(i_1, \ldots , i_r, j)$ with $i_1, \ldots , i_r, j \in \{1,\ldots ,m\}$, there is a Milnor invariant $\mu_{i_1,\ldots ,i_r, j}$.  Provided all the indices $(i_1,\ldots ,i_r,j)$ are distinct, $\mu_{i_1,\ldots ,i_r, j}$ is a link homotopy invariant of $m$-component string links \cite{Milnor1954}. 
The invariant $\mu_{i_1,\ldots ,i_r, j}$ is defined as follows.  The fundamental group $\pi_L$ of the complement of the link $L$ is generated by the meridians $M_1, \dots, M_m$ of the components and their conjugates.  The quotient by the  $q$-th term $\pi_L^{(q)}$ in the lower central series of $\pi_L$ is generated by the $M_i$ themselves \cite{MilnorIsotopy}.  So in this quotient, the $j$-th longitude $\ell_j$ can be written as a word $w_j$ in the $M_i$.  One then considers the Magnus expansion, which is a homomorphism from the free group $F\langle M_1, \dots, M_m\rangle$ to the ring  $\Z \langle \langle t_1, \ldots, t_m \rangle\rangle$ of power series in non-commuting variables $t_1,\ldots ,t_m$.  This map sends $M_i$ to $1+t_i$ and $m_i^{-1}$ to $1-t_i +t_i^2 -\cdots$.  Finally, one considers the image of $w_j$ under this map and takes $\mu_{i_1,\ldots ,i_r, j}$ to be the coefficient of $t_{i_1} \cdots t_{i_r}$ in this power series.  This number is well defined, provided one works modulo $\pi_L^{(q)}$ with $q\geq r$.\footnote{In fact, one can also define the Milnor invariants as elements dual to the Lie algebra of the lower central series of the link group.}
In the setting of closed links, as in Milnor's original work, the construction is similar, except that there is a choice for each meridian which may affect the result, and thus
each invariant is defined only up to the gcd of the lower-order invariants.
For string links, however, there is no indeterminacy, so the invariants are well defined integers.
For $m$-component links in either setting, the dimension of the space of Milnor link homotopy invariants of length $m$ (i.e.~with $m$ indices) is $(m-2)!$ \cite{Milnor1954}.

The invariant $\mu_{i_1,\ldots ,i_r, j}$ is finite type of order $r$, one less than the length of the invariant.  There are several proofs in the literature \cite{BarNatanJKTR, MellorMilnorWeight}.
From now on, we will use the term ``Milnor invariant'' to refer to Milnor invariants with distinct indices, i.e.~the link homotopy invariants.
The simplest example is the case $r=1$, where $\mu_{i;j}$ is the pairwise linking number of the $i$-th and $j$-th strands, an invariant of finite type 1.  The next example is that of the \emph{triple linking number} $\mu_{i_1, i_2;j}$
which has been the subject of much investigation, especially in recent years \cite{MunsonAGT, MunsonJTop, DGKMSV1, DGKMSV2, MellorMelvin}. 
We will revisit both of these invariants
in Section \ref{Examples} from the configuration space integral point of view.

Each invariant $\mu_{i_1,\ldots ,i_r; j}$ is also an invariant of the $(r+1)$-component sublink determined by the strands labeled $i_1,\ldots ,i_r,j$.  Thus to study a given Milnor invariant of type $n$, we may restrict to type $n$ invariants of $(n+1)$-component string links, and from now on we assume $m=n+1$.

\section{Examples in low degrees}
\label{Examples}

Before proving our theorem, we will illustrate it with examples in the several lowest orders.  The idea is 
as follows: Theorem \ref{FiltrationIsomorphisms} will recover the correspondence between Milnor invariants and linear combinations of trivalent trees.  Given a Milnor invariant, such a linear combination of trees is part of the integral formula for that invariant.  The STU conditions are not quite enough to determine the rest of the formula, but they are enough to determine the formula up to products of lower-order invariants.  (The products of lower-order invariants correspond to cocycles in the graph complex which can be indexed by linear combinations of \emph{forests}.)

\subsection{The pairwise linking number}
The only type 1 invariant (of unframed long links) is the pairwise linking number.  This invariant corresponds to a diagram with one chord with one endpoint on each of two segments.  Thus, at this order, the correspondence between link homotopy invariants and trivalent trees is clear.  The linking number can then be written as the integral associated to this one diagram, and this is essentially the Gauss linking integral (though for long links instead of closed links).  That is, the linking number of strands $i$ and $j$ of a link $L$ is
$$
\operatorname{lk}(L_i,L_j)=(I_{\mathcal{H}})_{\Gamma}(L)=(I_{\mathcal{L}})_{\Gamma}(L)=\int\limits_{C[1,1; \R\sqcup \R]} \left(\frac{L(x)-L(y)}{|L(x)-L(y)|}\right)^* \omega
$$
where $\omega$ is the volume form on $S^2$ and $C[1,1;\R\sqcup\R]$ is the compactified configuration space of two points, one on each copy of $\R$; this compactified configuration space is just an octagonal disk (see \cite[Section 1.2]{KoytcheffAGT2014} for details).

\subsection{The triple linking number}
\label{TripleLinking}
For our discussion of invariants of type 2 or higher, it will be convenient to recall Lemmas \ref{ProductOfFiniteTypes} and \ref{ShuffleProductLemma}, which respectively concern products of finite type invariants and products of graph cocycles.
Recall also from Section \ref{MilnorInvs} that the dimension of the space of Milnor (link homotopy) invariants of $n$-component links is $(n-2)!$ and that any such invariant is of type $n-1$.
Thus at finite type 2, the space of Milnor invariants of 3-component links is one-dimensional,
and we call a generator the triple linking number $\mu_{123}$.  Recall also that type $n$ invariants correspond to diagrams with $2n$ vertices.
It is easy to see that the space of trivalent trees on $2\cdot 2=4$ vertices with leaves on distinct segments (modulo the IHX relation)
also has one generator.  It is the tripod $\tau$, 
pictured below:
\[
\includegraphics[height=2.5pc]{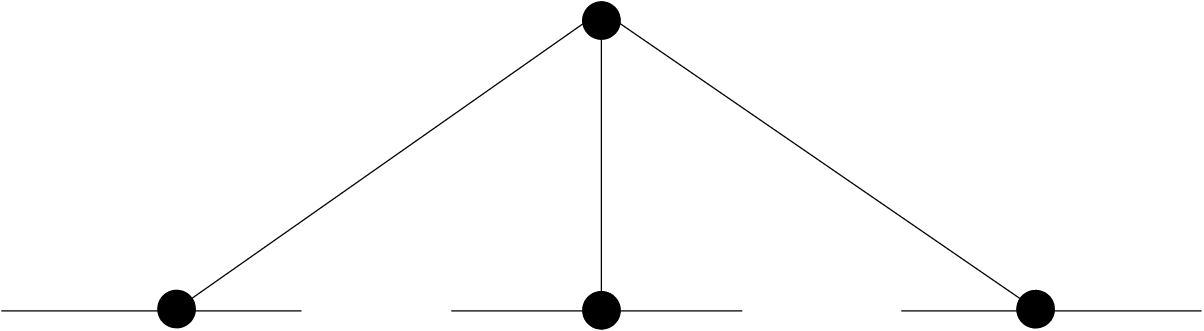}
\]
Based on the correspondence of Habegger and Masbaum between Milnor invariants and such trivalent trees, one could na\"{i}vely expect the triple linking number to be given by the integral associated to $\tau$.  However, the graph cocycle $\alpha_{123}$ for $\mu_{123}$ contains some chord diagrams as well.  Thus the configuration space integral for $\mu_{123}$ has more terms than just the integral associated to $\tau$.  
Mellor established a recursive formula \cite[Theorem 2]{MellorMilnorWeight} for calculating the weight system of any Milnor invariant.
In \cite{KoytcheffAGT2014}, the first author used this formula to calculate the values of the weight system $W_{123}$ for $\mu_{123}$, or equivalently the graph cocycle $\alpha_{123}$.  
A priori, there are, in addition to $\tau$, six link homotopy diagrams that touch all three segments and which one might thus expect to appear in $\alpha_{123}$.
They are the diagrams $L,M,R$, \\
\[
\includegraphics[height=1.6pc]{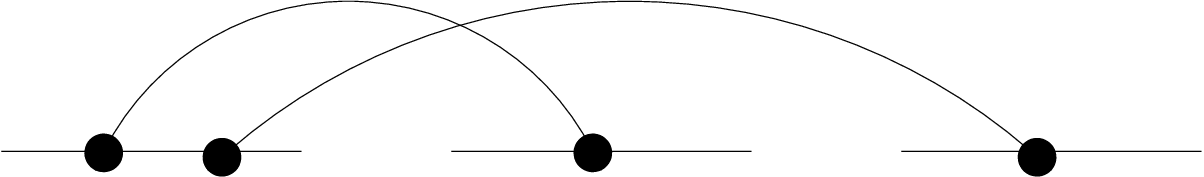},\ 
\includegraphics[height=1.6pc]{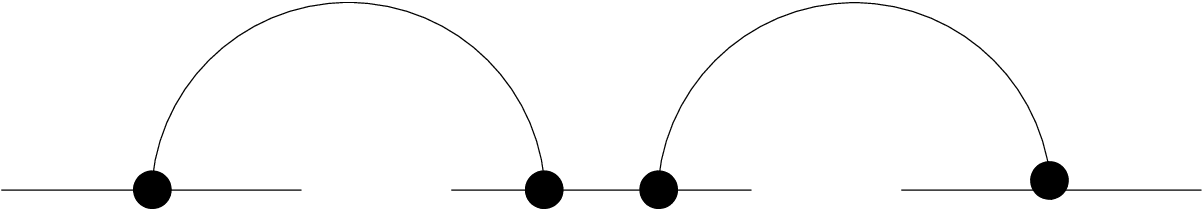},\ 
\includegraphics[height=1.6pc]{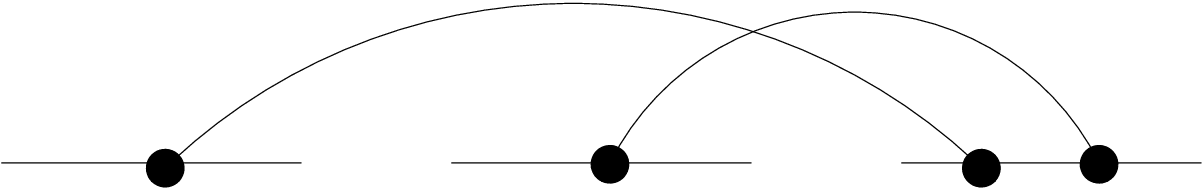}
\]

and the diagrams $L',M',R'$, \\
\[
\includegraphics[height=2.6pc]{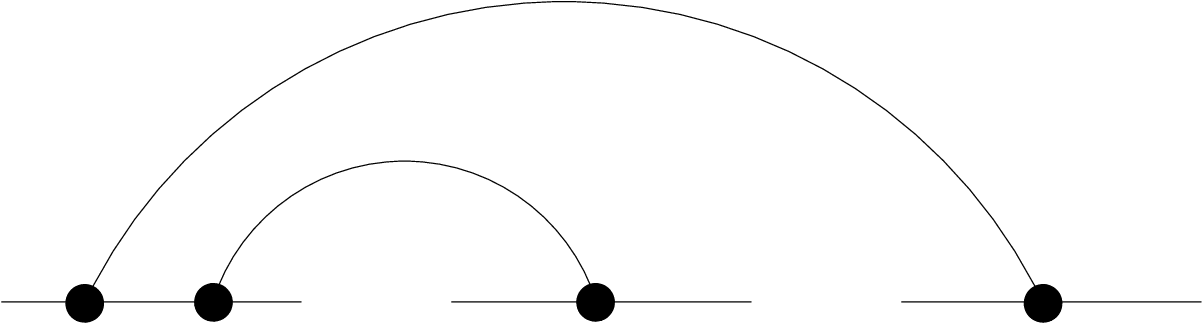},\ 
\includegraphics[height=1.6pc]{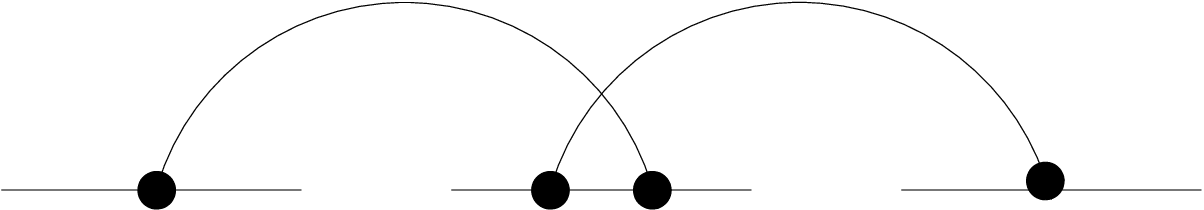},\ 
\includegraphics[height=2.6pc]{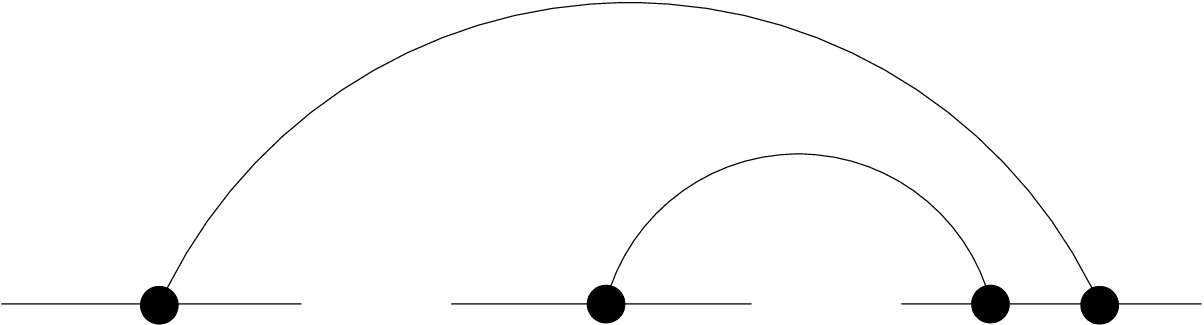}.
\]
Applying Mellor's formula shows that $\alpha_{123}=L-M+R-\tau$ \cite[Section 3.1]{KoytcheffAGT2014}.  (One can also use Mellor's formula to check that no other (chord) diagrams of order 2 appear in $\alpha_{123}$.)

Thus in terms of configuration space integrals,
\begin{equation}
\label{FormulaForMu123}
\mu_{123} =  I_L - I_M + I_R - I_\tau 
\end{equation}
where the terms $I_\Gamma$ above are the integrals over compactified configuration spaces associated to the diagrams $\Gamma=L,M,R,\tau$ as in \eqref{GeneralIntegral}:
\footnote{Erratum to \cite{KoytcheffAGT2014}:
The minus signs above come from the correspondence between Lie orientations and ``integration orientations.''  In \cite{KoytcheffAGT2014}, these minus signs were missing.  This was due to an erroneous identification of Lie orientations with integration orientations in the paragraph before Remark 3.1 of that paper.  
(The signs in the two versions of the STU relation in that paper are correct, as are the calculated values of $W_{123}$ on the Lie-oriented diagrams.) 
}
\begin{align*}
I_L = -\int_{F[2,1,1;0]} \theta_{13}\theta_{24} & &
I_M= \int_{F[1,2,1;0]} \theta_{12}\theta_{34}\\
I_R= -\int_{F[1,1,2;0]} \theta_{13}\theta_{24} & &
I_\tau= -\int_{F[1,1,1;1]} \theta_{14}\theta_{24}\theta_{34}
\end{align*}

From the geometry of configuration space integrals, it makes sense that this is the formula for the triple linking number.  Namely, the configuration space captured schematically by $\tau$ -- four points in $\R^3$, three of which are constrained to lie on the three strands of the long link -- has three boundary components given by the free point colliding with the points on the link strands.  There is no reason for the integral $I_\tau$ over these boundary components to vanish, so the integral is not necessarily an invariant.  However, each of the spaces corresponding to diagrams $L$, $M$, and $R$ also has a boundary component given by the collision of two points on the same strand.  The restrictions of the integrals $I_L$, $I_M$, and $I_R$ to those boundaries thus precisely cancel the three boundary contributions from $I_\tau$ in the sum above.

A further evidence that it is necessary to consider all four integrals is given by the fact that it is not hard to find a long link (even one with zero pairwise linking numbers) for which at least one of the integrals associated to $L,M,R$ does not vanish.  So the configuration space integral for $\mu_{123}$ is not merely $I_\tau$ or $-I_\tau$ even in special cases.  Nor do the STU conditions give the coefficients of the other diagrams from the coefficient of $\tau$.  However, one can ultimately conclude in this low-order example that the coefficient of $\tau$ together with the STU condition \emph{does} determine the coefficients of $L,M,R,L',M',R'$ \emph{up to linear combinations of} $\{L+L', M+M', R+R'\}$.  By Lemma \ref{ShuffleProductLemma}, these three elements correspond (up to sign) to the products of pairwise linking numbers.  This is the first nontrivial example of our more general main result, and the ``up to products of pairwise linking numbers'' caveat is a good indicator of the type of statement we will establish in Section \ref{Results}.

Contrast formula (\ref{FormulaForMu123}) with the formula of Habegger and Masbaum \cite{HabeggerMasbaum}, which says that for a link with no pairwise linking, $\mu_{123}$ is given by just the Kontsevich integral for the tripod, i.e.~the coefficient of $\tau$ in the Kontsevich integral.  This discrepancy is likely due to the fact that they consider $\tau$ as an element in the \emph{quotient} of the diagram space by the STU relations.  
One could dualize to the corresponding ``STU subspace'' of graph cocycles (i.e.~find the orthogonal complement to the STU relations) and ultimately reach the conclusion above regarding the possible diagrams in this example of $\alpha_{123}$.  
However, the dualizing isomorphism is not canonical, and for the general case it is easier to work with the graph cocycles directly.

\subsection{The quadruple linking numbers}
\label{QuadrupleLinking}
At finite type 3, Mellor's formula reveals that the graph cocycle $\alpha_{1234}$ for the invariant $\mu_{1234}$ has 24 terms, including 13 of the 72 chord diagrams,
9 of the 24 diagrams with one free vertex, and 2 of the 3 four-leaved trivalent trees.  This points to the potential difficulty of quickly finding the integral formula for an arbitrary Milnor invariant.  
Note however that there are only 3 isomorphism classes of trivalent trees with four leaves, distinctly labeled by 1, 2, 3, 4.  The IHX relation says that the (signed) sum of these is zero, so the dimension of the space of these trees modulo IHX is 2.  This is also the number of type 3 Milnor invariants of 4-component long links.  Our main result will say that a set of coefficients of these trees which satisfy the IHX condition almost determines a sum of type 3 Milnor invariants; in fact, these coefficients determine it up to adding type 2 invariants, namely triple linking numbers and products of pairwise linking numbers.

\section{Main Results}
\label{Results}
To obtain our main results, we start by filtering the spaces of weight systems by connected components in Section \ref{FilterWSec}.  This yields our first main result, namely a decomposition of these spaces in terms of the shuffle product.  We specialize to homotopy invariants in Section \ref{HomotopyWSec}.  We define similar filtrations on the space of invariants using products of Milnor invariants in Section \ref{FilterHVSec}, and we tie these together in Section \ref{FiltrationIsoSection} to obtain our remaining main results.  Finally we extend to cohomology of homotopy links in higher-dimensional space in Section \ref{Cohomology}.

\subsection{The shuffle product and filtrations of weight systems for finite type link invariants}
\label{FilterWSec}


By a \emph{connected} diagram we mean a diagram which is connected after removing the link strands.  By a \emph{connected component} of a diagram, we similarly mean a component of the result of removing the link strands.

\begin{definition}
\label{FilterByComponentsDef}
Let $\LW^k_n$ (resp.~$\HW^k_n$) be the subspace of $\LW_n$ (resp.~$\HW_n$) spanned by diagrams with 
at least $k$ connected components.
\end{definition}

It is clear that 
\[
\{0\} = \LW^{n+1}_n \subset \LW^n_n \subset \cdots \subset \LW^1_n = \LW
\]
and that the shuffle product respects this filtration in the sense that 
$$\bullet: \LW^k_n \otimes \LW^\ell_p \longrightarrow \LW^{k+\ell}_{n+p}.$$


\begin{definitions}
\label{ConnectedIHXDef}
Let $\mathcal{J} \subset \LD_{0,\ast}$ be the subspace spanned by linear combinations of connected trivalent diagrams 
which satisfy the IHX conditions.  
(Note that the terms $I,H,X$ in each condition all have the same number of components.)  
Denote by $S(\mathcal{J})$ the symmetric algebra (i.e.~polynomial algebra) on $\mathcal{J}$, and view a product therein as a disjoint union of connected components.
Let $S^k(\mathcal{J})_n \subset S^k(\mathcal{J})$ be the subspace of elements (with $k$ components) of order $n$, i.e.~$2n$ vertices in total.  
\end{definitions}

Clearly the multiplication in $S(\mathcal{J})$ restricts to $$S^k(\mathcal{J})_n \otimes S^\ell(\mathcal{J})_p \longrightarrow S^{k+\ell}(\mathcal{J})_{n+p}.$$  

A cocycle $\alpha \in \LW_n$ can be written as a linear combination of diagrams with various numbers of components; call the diagrams with the fewest components the \emph{leading} terms of $\alpha$.
Note that these leading terms by themselves are generally not cocycles.

This is the main result of this subsection:
\begin{theorem}
\label{LWSubalgebra}
Taking the leading terms of a cocycle defines injective maps  
$$\LW^k_n / \LW^{k+1}_n \incl S^k(\mathcal{J})_n.$$  
Moreover, these induce an injective algebra map $\LW \incl S(\mathcal{J})$ which takes the shuffle product on $\LW$ to multiplication in $S(\mathcal{J})$.
\end{theorem}

\begin{remarks} \
\begin{enumerate}
\item  The map $\LW \incl S(\mathcal{J})$ does not encode the differential structure, and elements in $\mathcal{J}$ are generally not graph cocycles.  For example, as mentioned in Section \ref{Examples}, the cocycle for the triple linking number is $\alpha_{123}:=L-M+R-\tau$.  Its leading term $-\tau$ (the tripod) is in $\mathcal{J}$, but $\alpha_{123}$ is neither an element of $\mathcal{J}$ nor a shuffle product of such elements.
\item  This map will depend on the choices (for various $n$) of splittings $$\LW_n \cong \LW^n_n \oplus \LW^{n-1}_n / \LW^n_n \oplus \cdots \oplus \LW^1_n / \LW^2_n.$$
\item  This map is not surjective because we have not defined $\mathcal{J}$ to take into account STU conditions involving connected diagrams.  For example, at order 3, there are combinations of trees satisfying IHX (and thus in $\mathcal{J}_3$), but there are no weight systems for long knots whose leading terms are these elements of $\mathcal{J}$ (see \cite[Section 5.1, Figure 3]{CCRL-AGT}).  However, the IHX conditions suffice for our analysis of homotopy invariants.
\item
In Section \ref{FiltrationIsoSection} we will see that $\HW$ maps isomorphically onto the polynomial algebra generated by the homotopy link diagrams in $\mathcal{J}$, which can be viewed as trees with distinctly labeled leaves.  
(The splitting of $\HW$, and hence this isomorphism, will also be fixed by using the Milnor invariants.)
Based on the results of \cite{HabeggerMasbaum}, we suspect that similar results can be obtained for finite type concordance invariants or finite type invariants of pure braids using trees with some repetition allowed among the leaf-labels.  
\item
In a previous draft, we used a filtration by the number of free vertices.  While (up to reindexing) that filtration coincides with the one by components for $\HW$, the filtration by free vertices is finer for $\LW$.  Nonetheless, it would not strengthen our main results for the algebras as a whole.
\end{enumerate}
\end{remarks}


We need a few lemmas to prove Theorem \ref{LWSubalgebra}.  

\begin{lemma}
\label{NothingButShuffleProducts}
For any cocycle $\alpha \in \LW_n$, the leading terms of $\alpha$ are in the span of shuffle products of connected diagrams. 
\end{lemma}

\begin{proof}
Let $\alpha = \sum a_i \Gamma_i$ and let $k$ be the minimum number of components among the $\Gamma_i$.  That is, $k$ is the filtration of the leading terms of $\alpha$.  
We will partition the $\Gamma_i$ with $k$ components according to their underlying unitrivalent diagrams.  For such a diagram $\Gamma$, let $\mathcal{U}(\Gamma)$ be the set of all diagrams with the same underlying unitrivalent diagram as $\Gamma$.  Notice that $\mathcal{U}(\Gamma)$ contains all the diagrams which appear in the shuffle product of the connected components of $\Gamma$.  

Next construct for each $\mathcal{U}(\Gamma)$ an auxiliary ``STU graph,'' with a vertex for each diagram $\Gamma' \in \mathcal{U}(\Gamma)$ and an edge between $\Gamma'$ and $\Gamma''$ when there exists an STU condition involving these two diagrams.  (Since $\Gamma'$ and  $\Gamma''$ have the same number of free vertices, they must be the diagrams $T$ and $U$.)  Notice that for each $\mathcal{U}(\Gamma)$, this STU graph is connected; in fact, any two diagrams in the same equivalence class are related by a sequence of moves, where each move swaps adjacent legs (leaves) on the same segment, and in each such move, the two diagrams are $T$ and $U$ in some STU condition.  As a result, STU together with the fact that $\alpha$ has no diagrams with less than $k$ components imply that for each $\mathcal{U}(\Gamma)$, the coefficients of \emph{all} the diagrams in $\mathcal{U}(\Gamma)$ are determined by any one of them.  Thus the leading terms of $\alpha$ must be shuffle products of connected diagrams.

(In fact, the coefficients of the diagrams in each $\mathcal{U}(\Gamma)$ are all equal, since for any diagrams $T$ and $U$ as in the STU condition, $|\Aut(T)|=|\Aut(U)|$; thus they match the coefficients of the shuffle product of the connected components of $\Gamma$.  Alternatively -- without explicitly establishing the relationship between the coefficients -- the above argument shows that the space of coefficients for the various shuffles of the components of $\Gamma$ is one-dimensional, so up to a scalar multiple, the coefficients appearing in the shuffle product are the only possibility.)
\end{proof}

This proof breaks down for cocycles of defect greater than zero because the analogues of the STU conditions for such cocycles involve more than just swaps of segment vertices.

\begin{lemma}
\label{ConnectedCocyclesIHX}
For any cocycle $\alpha\in \LW_n$, the leading terms of $\alpha$ are in $S(\mathcal{J})$, the span of shuffle products of linear combinations of connected diagrams satisfying the IHX conditions.
\end{lemma}

We will prove this via the following
basic linear algebra fact:

\begin{lemma}
\label{SymmetricPowerLemma}
Let $V$ be a finite-dimensional 
vector space with a basis
and the resulting inner product.  Let $U$ be a subspace.  For $m\geq 1$, equip the $k$-fold symmetric power $S^k(V)$ with the inner product associated to the basis for $S^k(V)$ given by monomials in the basis for $V$.  Let $(U)$ denote the 
subspace in $S^k(V)$ generated by $U$, i.e.~the image of $U\otimes V^{\otimes(k-1)}$ in $S^k(V)$.  Then 
\[
(U)^\perp \cong S^k(U^\perp).
\]
\end{lemma}
\begin{proof}
By writing $V= U \oplus U^\perp$, it is easy to see that $S^k(V) = (U) \oplus S^k(U^\perp)$.
\end{proof}

\begin{proof}[Proof of Lemma \ref{ConnectedCocyclesIHX}]
Let $\mathcal{C}$ be the span of all connected diagrams.  Let $U$ denote the subspace of IHX combinations in $\mathcal{C}$, so that $\mathcal{J} = U^\perp \subset \mathcal{C}$.  By Lemma \ref{NothingButShuffleProducts}, the leading terms of a cocycle $\alpha \in \LW$ are contained in $S^k(\mathcal{C})$ for some $k$, where a polynomial in elements of $\mathcal{C}$ is identified with the corresponding shuffle product of diagrams.  The fact that a graph cocycle satisfies all IHX conditions (see Section \ref{DescribingCocycles}) means that the leading terms of $\alpha$ are contained in $(U)^\perp \subset S(\mathcal{C})$.  By Lemma \ref{SymmetricPowerLemma}, $(U)^\perp = S(\mathcal{J})$.
\end{proof}

\begin{proof}[Proof of Theorem \ref{LWSubalgebra}]
Define a map 
$$\LW^k_n / \LW^{k+1}_n \longrightarrow S(\mathcal{J})^k_n$$ by sending a graph cocycle to its terms with exactly $k$ components.  From Lemma \ref{ConnectedCocyclesIHX} and the definitions of the spaces, it is clear that this map is well defined, linear, and injective.  This proves the first statement of Theorem \ref{LWSubalgebra}.

For the second statement, note that the shuffle product descends to the filtration quotients:
$$\bullet\colon \LW^k_n / \LW^{k+1}_n \otimes \LW^\ell_p / \LW^{\ell+1}_p \longrightarrow\LW^{k+\ell}_{n+p} / \LW^{k+\ell +1}_{n+p}.$$
As noted earlier, the product in $S(\mathcal{J})$ restricts to $$S^k(\mathcal{J})_n \otimes S^\ell(\mathcal{J})_p \to S^{k+\ell}(\mathcal{J})_{n+p}.$$ 

Canonically, $S(\mathcal{J}) \cong \bigoplus_{k,n} S^k(\mathcal{J})_n$.  
By choosing a splitting of $\LW$ into summands of various orders $n$ and filtrations $k$, we obtain a map $\LW \to S(\mathcal{J})$.
Since the leading terms of a shuffle product $\alpha \bullet \beta$ are the shuffle product of the leading terms of $\alpha$ and $\beta$, this map is an algebra homomorphism.  
\end{proof}

\subsection{Specializing to homotopy weight systems}
\label{HomotopyWSec}
Recall that $\HW_n \subset \LW_n$ is the subspace of cocycles where every diagram has no closed paths.  Thus a diagram in $\HW_n$ can be viewed as a (Lie-oriented) trivalent forest with labeled leaves, where segment vertices correspond to leaves, free vertices correspond to trivalent vertices, and labels correspond to link strands.  The reader may easily verify the following:

\begin{lemma}
\label{FreeVerticesInTrivalentForest}
A trivalent tree on $2n$ vertices has exactly $n-1$ trivalent vertices.  
\qed
\end{lemma}

As with $\LW$, we have
$$\{0\} = \HW^{n+1}_n \subset \HW^n_n \subset \HW^{n-1}_n \subset \cdots \subset \HW^{1}_n = \HW_n.$$  

Let $\mathcal{T} \subset \HD$ be the intersection of $\mathcal{J} \subset \LD$ with $\HD$.  Thus $\mathcal{T}$ is the subspace of linear combinations of diagrams of (Lie-oriented) \emph{trees} with \emph{distinctly} labeled leaves, which satisfy the IHX conditions.  

Define $S^k(\mathcal{T})_n$ as the subspace of the $k$-th symmetric power $S^k(\mathcal{T})$ spanned by order-$n$ products, i.e.~forests with $2n$ vertices and $k$ components.  
The restriction of the map of Theorem \ref{LWSubalgebra} to $\HW$ clearly lands in $S(\mathcal{T})$:

\begin{corollary}
\label{ShuffleProductsSpanHW}
The space $\HW^k_n / \HW^{k+1}_n$ injects into the space of linear combinations of forests with $k$ components which satisfy the IHX conditions:
$$\HW^k_n / \HW^{k+1}_n \longrightarrow S^k(\mathcal{T})_n.$$
There is an injective map of algebras $\HW \to S(\mathcal{T})$.
\qed 
\end{corollary} 

We will show in Section \ref{FiltrationIsoSection} that these maps are in fact isomorphisms.

\begin{remark}
As noted for Theorem \ref{LWSubalgebra}, these maps are given by taking the leading terms, and the leading terms alone are in general not cocycles.
\end{remark}

In particular Corollary \ref{ShuffleProductsSpanHW} implies that $\HW_n / \HW^{2}_n = \HW^{1}_n / \HW^{2}_n$ maps to $\mathcal{T}$ itself.  Its image is contained in the subspace $\mathcal{T}_n \subset \mathcal{T}$ of order-$n$ elements, which are by Lemma \ref{FreeVerticesInTrivalentForest} certain linear combinations of (Lie-oriented) trees with $n+1$ distinctly labeled leaves.  As mentioned in Section \ref{MilnorInvs}, it suffices to consider links of $n+1$ components to study type $n$ Milnor invariants; in this case the labeling is a bijection with $\{1,\ldots ,n+1\}$.

The next Lemma can be verified by considering the dual space of trees \emph{modulo} the IHX relation, or equivalently a space of Lie bracket expressions modulo the Jacobi identity.  See for example \cite[Lemma 1.10]{Sinha-Graphs-Trees} or \cite[Theorem 3]{BarNatanJKTR}.

\begin{lemma}
\label{TrivalentTrees}
In the setting of links of $n+1$ components, the dimension of $\mathcal{T}_n$ is $(n-1)!$.
\qed
\end{lemma}

\subsection{Filtering homotopy link invariants}
\label{FilterHVSec}
Recall that $\HV_n$ is the vector space of type $n$ link homotopy invariants of long links on $n+1$ strands.  This space contains type $n$ Milnor invariants, as well as products of lower-order invariants, as guaranteed by Lemma \ref{ProductOfFiniteTypes}.  We now define subspaces of $\HV_n$ which will correspond to our filtration stages of $\HW_n$.  

\begin{definition}
For any $k=1,\ldots ,n$, let $\HV^k_n$ be the subspace of $\HV_n / \HV_{n-1}$ spanned by 
(residues of) products of at least $k$ Milnor invariants.  (Here we do not consider the type 0 constant invariant to be a Milnor invariant.)
\end{definition}

Since $\HV^k_n$ is a subspace of $\HV_n / \HV_{n-1}$, we may take the generating elements to be type $n$ products of at least $k$ Milnor invariants.
So if $n_1,\ldots ,n_{\ell}$ are the orders of the $\ell$ invariants in such a product (with $\ell \geq k$), then Lemma \ref{ProductOfFiniteTypes} implies $n_1+\cdots+n_\ell=n$.  
Thus 
$$\{0\} = \HV^{n+1}_n \subset \HV^n_n \subset \cdots \subset \HV^1_n \subset \HV_n / \HV_{n-1}.$$  
A consequence of Corollary \ref{PolynomialOnMilnorInvts} below is that this is in fact a filtration of $\HV_n / \HV_{n-1}$.

\subsection{The filtrations and the canonical map}
\label{FiltrationIsoSection}
We now relate the two filtrations we have defined for homotopy weight systems and homotopy invariants of finite type.

\begin{theorem}
\label{FiltrationIsomorphisms}
For every $k=0,1,\ldots ,n-1$, the isomorphism 
\[
\xymatrix{\HV_n / \HV_{n-1} \ar[r]^-\cong &  \HW_n}
\]
from Theorem \ref{FTFTI} -- given by the canonical map \eqref{CanonicalMap}, and with inverse given by configuration space integrals \eqref{InverseInFTFTI} -- restricts to an isomorphism of subspaces
\[
\HV^{k}_n \cong  \HW^k_n.
\]
\end{theorem}

\begin{proof}
A basis element of $\HV^{k}_n$ is represented by a product $v= v_1 \cdots v_\ell$ of $\ell \geq k$ Milnor invariants.  Since configuration space integrals take the shuffle product to the product of invariants, there are graph cocycles $\alpha_1,\ldots ,\alpha_\ell$ such that $\alpha:=\alpha_1 \bullet \cdots \bullet \alpha_\ell$ maps to $v$ (at least in the quotient $\HV_n / \HV_{n-1}$).  
Thus $\alpha \in \HW^k_n$.  Since the configuration space integral map is inverse to the canonical map, the latter takes $\HV^{k}_n$ (injectively) to $\HW^k_n$:
\begin{equation}
\label{FiltIsoEqn}
\xymatrix@C-2.2pc@R-0.6pc{
\{0\} = & \HV^{n+1}_n \ar[d] & \subset \HV^n_n \ar[d] & \subset \HV^{n-1}_n \ar[d] & \subset \cdots \subset \negthinspace & \HV^1_n \ar[d]  & \subset \HV_n / \HV_{n-1} \\ 
\{0\} =: & \HW^{n+1}_n & \subset \HW^n_n & \subset \HW^{n-1}_n & \subset \cdots \subset \negthinspace &  \HW^{1}_n &  =\HW_n \qquad \quad \,\,
}
\end{equation}
To show that each such map is an isomorphism, we proceed by downward induction on $k$, where the basis case of filtration stage $n+1$ is trivially true.  So assume the map is an isomorphism at filtration stage $k+1$, and consider filtration stage $k$.
The above diagram implies we get a map on filtration quotients $\HV^{k}_n / \HV^{k+1}_n \to \HW^k_n / \HW^{k+1}_n$.  
The induction hypothesis together with the injectivity of the canonical map implies that this map is injective.  

For surjectivity, it suffices to count dimensions.  A basis for $\HV^{k}_n / \HV^{k+1}_n$ is given by products $v_1 \cdots v_{k}$ such that  
\begin{itemize} 
\item each $v_i$ is a Milnor invariants of $(n+1)$-component long links, and
\item $\sum (\ell(v_i)-1)=n$, where $\ell(v_i)$ is the length of the Milnor invariant $v_i$.  
\end{itemize}
By Corollary \ref{ShuffleProductsSpanHW}, $\HW^k_n / \HW^{k+1}_n$ injects into $S(\mathcal{T})^k_n$.  A basis for the latter space is given by (shuffle) products $\tau_1 \bullet \cdots \bullet \tau_{k}$ such that 
\begin{itemize}
\item each $\tau_i$ is a linear combination of Lie-oriented trees, with leaves distinctly labeled by $\{1,\ldots ,n+1\}$, satisfying the IHX conditions, and 
\item $\sum (\ell(\tau_i) -1) = n$, where $\ell(\tau_i)$ is the number of leaves of $\tau_i$.  
\end{itemize} 
The number of eligible $v_i$ with $\ell(v_i)=\ell$ is $(\ell-1)!$, as noted in Section \ref{MilnorInvs}.
The same is true of the number of eligible $\tau_i$ with $\ell(\tau_i)=\ell$ by Lemma \ref{TrivalentTrees}.  Thus in the map $\HV^{k}_n / \HV^{k+1}_n \to \HW^k_n / \HW^{k+1}_n$, the dimension of the codomain is at most the dimension of the domain, hence the map is an isomorphism.  
By induction on $k$ and the five lemma, these isomorphisms of filtration quotients yield the desired isomorphisms $\HV^{k}_n \cong \HW^k_n$.
\end{proof}

Note that the last step of the proof above also proves the following statement: 

\begin{corollary}
\label{HWLeadTermsIsIso}
The injective maps $\HW^k_n / \HW^{k+1}_n \to S^k(\mathcal{T})_n$ in Corollary \ref{ShuffleProductsSpanHW}, given by taking leading terms, are in fact isomorphisms.
\qed
\end{corollary}


\begin{corollary}
\label{MilnorInvariantsTrivalentTrees}
The space of type $n$ Milnor invariants is isomorphic to the space $\mathcal{T}_n$ of Lie-oriented trivalent trees with $n+1$ distinctly labeled leaves which satisfy the IHX conditions.
\end{corollary}
\begin{proof}
Consider the composite below:
\[
\HV^1_n/\HV^2_n \longrightarrow \HW^{1}_n/\HW^{2}_n \longrightarrow S^1(\mathcal{T})_n = \mathcal{T}_n
\]
The first map (induced by the canonical map) is well-defined and an isomorphism by Theorem \ref{FiltrationIsomorphisms} with $k=1$ and $k=2$.  The second map (which takes leading terms) is an isomorphism by Corollary \ref{HWLeadTermsIsIso}.
Now just note that the left side has a basis represented by type $n$ Milnor invariants.
\end{proof}
  
\begin{remarks}\label{Dualizing} \
\begin{enumerate}  
\item[(a)] 
By dualizing from weight systems to diagrams, one can also describe this space as the \emph{quotient} $\mathcal{T}_n^*$ of such trees by the IHX relations, as often appears in the literature.
This correspondence goes back at least as far as Habegger and Masbaum \cite[Theorem 6.1, Proposition 10.6]{HabeggerMasbaum}.  
Essentially, this quotient of the space of diagrams corresponds to homology, whereas weight systems correspond to cohomology.  
The algebra of weight systems is thus better suited to
configuration space integrals, and it brings us closer to being able to write formulae for Milnor invariants in terms of those integrals.
%
The fact that we work with filtration quotients $\HW^k_n / \HW^{k+1}_n$ above seems to correspond to the ``first non-vanishing Milnor invariant" caveat in Habegger and Masbaum's Theorem 6.1. 
\item[(b)]
Other authors have used a co-commutative Hopf algebra structure on the space of diagrams (dual to weight systems).  The space of weight systems also has a Hopf algebra structure, with the coproduct coming from stacking string links and hence deconcatenation of diagrams.  However, this Hopf algebra is not co-commutative.  Moreover, though it is possibly true, we have not yet proven that configuration space integrals for long links respects the coproduct.  
\end{enumerate}
\end{remarks}

\begin{corollary}
\label{PolynomialOnMilnorInvts}
The space $\HV$ of finite type homotopy invariants of long links is the polynomial algebra on the Milnor invariants.
\end{corollary}
\begin{proof}
By Theorem \ref{FiltrationIsomorphisms}, the rightmost map in \eqref{FiltIsoEqn} is an isomorphism.  Hence $\HV^1_n = \HV_n / \HV_{n-1}$, and every finite type invariant is a linear combination of products of Milnor invariants. 
\end{proof}
Note that the only facts about Milnor invariants that we have used are that they are of finite type and that the number of linearly independent products of them at a given finite type $n$ coincides with the dimension of the appropriate space of forests.

\begin{remark}
There are other proofs of Corollary \ref{PolynomialOnMilnorInvts}.  For example, by considering the space dual to $\HW$, Bar-Natan uses a co-commutative Hopf algebra structure and an analogue of the Poincar\'{e}--Birkhoff--Witt correspondence \cite[Theorem 3]{BarNatanJKTR}.  In a similar vein, one could use the fact that finite type homotopy string link invariants correspond to the universal enveloping algebra of the Lie algebra of the lower central series of the link group (to which Milnor invariants are dual); the case of pure braids is addressed in the Ph.D.~thesis of Willerton \cite[Chapter 4.3]{WillertonThesis}.
\end{remark}

\begin{theorem}
\label{AlgIsos}
The Milnor invariants give canonical isomorphisms of algebras $\HV \cong \HW \cong S(\mathcal{T})$, using the products of  invariants on $\HV$, the shuffle product on $\HW$, and the usual product on $S(\mathcal{T})$.  
\end{theorem}
\begin{proof}
Consider the following diagram, where the isomorphisms in the top row come from Theorem \ref{FiltrationIsomorphisms} and Corollary \ref{HWLeadTermsIsIso} on each summand. 
\[
\xymatrix@C-0.5pc@R-0.2pc{
\bigoplus_{n,k} \HV^k_n / \HV^{k+1}_n  \ar[r]^-\cong \ar@{<->}[d]^{\cong} & 
\bigoplus_{n,k} \HW^k_n / \HW^{k+1}_n \ar[r]^-\cong \ar@{<->}[d]^{\cong}& 
\bigoplus_{n,k} S^k(\mathcal{T})_n \ar@{<->}[d]^{\cong} \\
\HV  \ar[r]^-{\cong} & \HW  \ar[r]^-{\cong} & S(\mathcal{T})
}
\]
By Corollary \ref{PolynomialOnMilnorInvts} the Milnor invariants give a canonical decomposition of $\HV$ into the direct sum above it.  A polynomial in Milnor invariants gives a well defined element of $\HW$ (via the canonical identification $\HW \cong \bigoplus_n \HW_n$), and we can use this to obtain the middle vertical arrow.  The symmetric algebra $S(\mathcal{T})$ can be canonically identified with the direct sum above it, and this determines the lower-right horizontal arrow.
We already know that the left-hand maps take the product of invariants to the shuffle product (since the inverse map of configuration space integrals is an algebra map).  We also know from Theorem \ref{LWSubalgebra} that the right-hand maps are algebra maps for any choice of splitting of $\HW$.  The only new twist is that we now have a canonical splitting using the Milnor invariants.
\end{proof}

By tracing the definitions of the maps involved, we may deduce from Theorem \ref{AlgIsos} the following partial recipe for constructing Milnor invariants via configuration space integrals:
\begin{corollary}
\label{PartialRecipe}
Suppose $\alpha$ is a linear combination of (leaf-labeled, Lie-oriented) trivalent forests satisfying the IHX condition such that each forest has  $k$ components with $n_1+1,\ldots ,n_k+1$ leaves respectively.  Then one can obtain a product of (a linear combination of) Milnor invariants of types $n_1,\ldots ,n_k$ (equivalently, lengths $n_1+1,\ldots ,n_k+1$) by the following procedure: 
\begin{itemize}
\item complete $\alpha$ to a cocycle $\alpha'$ by adding terms with more components (using the STU relation), and then
\item take the configuration space integral associated to $\alpha'$.
\end{itemize}
All choices of the terms with more components yield the same invariant up to products of more than $k$ Milnor invariants.
\qed
\end{corollary}

\subsection{Cohomology classes in spaces of link maps}\label{Cohomology}
%
We will now observe that the graph cocycles which give rise to homotopy long link invariants in the classical dimension $d=3$ also yield cohomology classes of arbitrarily high degree in $\mathrm{Link}(\coprod_{i=1}^m \R, \R^d)$ for $d>3$.

\begin{theorem}
\label{HigherCohomology}
There is a subalgebra of cohomology classes in $\mathrm{Link}(\coprod_{i=1}^m \R, \R^d)$ isomorphic to the polynomial algebra on Milnor homotopy invariants.  
\end{theorem}

\begin{proof}[Sketch of proof]
Recall the integration maps 
 (\ref{CochainMap}),
\[
\mathcal{I}: \HD_{k,n} \longrightarrow \Omega_{dR}^{n(d-3)+k} \mathrm{Link}(\coprod_{i=1}^m \R, \R^d),
\]
where $k$ and $n$ are the defect and order and where tree diagrams on the left-hand side are given an orientation as discussed in Section \ref{Background}.  In particular, the left-hand side also depends on the parity of the ambient dimension $d$, even though it is suppressed from the notation.  
These integration maps are maps of differential algebras \cite[Theorem 4.33]{KMV}.  Now consider the induced maps on cohomology in defect zero for arbitrary ambient dimension $d$:
\begin{equation}
\label{DefectZeroArbitraryD}
\mathcal{I}: Z(\HD_{0,n}) \longrightarrow \mathrm{H}^{n(d-3)} \mathrm{Link}(\coprod_{i=1}^m \R, \R^d).
\end{equation}

By the same argument used in Section 2 and Theorem 6.1 of \cite{CCRL-AGT}, one can show that   (\ref{DefectZeroArbitraryD}) is injective for any $d$.  In fact, this argument is a generalization of the injectivity argument for $d=3$ used in the proof of Theorem \ref{FTFTI}.  
One first constructs a cycle corresponding to each chord diagram $\Gamma$; the cycle is an $n$-fold product of $(d-3)$-cycles, obtained by resolving double-points of a singular link in $\R^d$ corresponding to the chords in $\Gamma$.  Then one shows that any integral $I_{\Gamma'}$ is nonzero on this cycle if and only if $\Gamma'=\Gamma$.  Since every cocycle of trivalent diagrams has a chord diagram (see \cite[Proposition 5.1]{CCRL-AGT} for the case of knots), this ultimately establishes the injectivity of (\ref{DefectZeroArbitraryD}).

For $d=3$, the integration (\ref{DefectZeroArbitraryD}) maps isomorphically onto the finite type invariants \cite[Theorem 5.8]{KMV}.
We claim the left-hand side is actually independent of $d$.  
It clearly depends at most on the parity of $d$.
To show it does not depend on the parity, one can use the fact that for tree diagrams, orientations of odd and even type are in canonical correspondence.  The correspondence does not quite extend to arbitrary forest diagrams, but for forest diagrams in $\HD_{0,n}$, one can substitute a smaller space of ``braid diagrams'' where the trees have an ordering along the link strands and where the correspondence can thus be applied.  We refer the reader to our forthcoming work with Komendarczyk \cite{KKV-Braids-Integrals} for further details on this nuance. 
As a result, we can deduce from Corollary \ref{PolynomialOnMilnorInvts} that the space of graph cocycles $\HW_n \subset \HD_{0,n}$ is also isomorphic to the polynomial algebra on Milnor invariants.
Then applying the injectivity in cohomology,
 we get a subspace of $\mathrm{H}^*(\mathrm{Link}(\coprod_{i=1}^m \R, \R^d))$ isomorphic to the polynomial algebra on Milnor invariants.
\end{proof}

In particular, we get a \emph{nontrivial} cohomology class of degree $n(d-3)$ for every type $n$ product of Milnor invariants.  We will refer to such a cohomology class as a \emph{generalized Milnor invariant}.

We conjecture that the map \eqref{CochainMap} from the whole graph complex induces a surjection in cohomology.  We already know that in the case $d=3$, the restriction (\ref{DefectZeroArbitraryD}) induces an isomorphism on $\mathrm{H}^0$.  In fact, the Vassiliev conjecture holds in this case \cite{BarNatanJKTR}, i.e.~Vassiliev invariants separate homotopy string links.
It would also be interesting to determine for arbitrary $d$ how much of the cohomology of the space of link maps comes from the defect zero part of the graph complex.  Along these lines, we can loosely formulate the following:
 
\begin{conjecture}
All of the cohomology of the space of link maps $\mathrm{Link}(\coprod_{i=1}^m \R, \R^d)$, $d>3$, comes from integrals corresponding to products of Milnor invariants.
\end{conjecture}

We would also like to investigate geometric interpretations of our generalized Milnor invariants.  In particular, it would be interesting to see if these classes are related to the invariants of link maps $\mathrm{Link}(S^{d_1}\sqcup S^{d_2}\sqcup\cdots\sqcup S^{d_r}, \R^d)$ defined by Koschorke \cite{Koschorke97}.  These invariants are defined geometrically (they can be thought of as counting an overcrossing locus of a link) and can also be regarded as generalizations of Milnor invariants.  An interesting connection could also be established between Munson's work \cite{MunsonJTop}
on relating manifold calculus of functors to Koschorke's generalizations of Milnor invariants.  The connection between configuration space integrals for long links and manifold calculus of functors was established by the first author in \cite{Koytcheff-JHRS}
who factored the integrals for links through the Taylor tower of these spaces.  The functor calculus of Taylor towers could thus serve as a common ground for configuration space integrals and Koschorke's invariants and provide the context for relating them to each other.


\bibliographystyle{amsplain}
\bibliography{refs}

\end{document}